\documentclass[11pt]{amsart}
\usepackage[utf8]{inputenc}
\usepackage[nospace,noadjust]{cite}
\usepackage{relsize}
\usepackage{cite}
\usepackage{color}
\usepackage[dvipsnames]{xcolor}
\usepackage{marginfix}
\usepackage{cancel,marginnote}

\usepackage{tikz}
\usetikzlibrary{arrows.meta, positioning}

\usepackage{tikz-cd}
\usetikzlibrary{arrows.meta, calc, decorations.markings}
\usepackage{colortbl}
\usepackage{hyperref, enumitem}
\hypersetup{
	colorlinks   = true, %Colours links instead of ugly boxes
	urlcolor     = blue, %Colour for external hyperlinks
	linkcolor    = Purple, %Colour of internal links%
	citecolor   = red %Colour of citations
}
\usepackage[margin=2.5cm]{geometry}

\usepackage{comment}
\usepackage{hyperref}
\usepackage{cleveref,amssymb}

\usepackage{amscd}

\DeclareMathOperator{\C}{\mathcal{C}}

\newtheorem{theorem}{Theorem}[section]

\newtheorem{lemma}[theorem]{Lemma}

\newtheorem{corollary}[theorem]{Corollary}
\newtheorem{definition}[theorem]{Definition}
\newtheorem{proposition}[theorem]{Proposition}

\newtheorem{remark}[theorem]{Remark}

\usepackage{mdframed}

\usepackage{url}

\newcommand{\fqn}{\mathbb{F}_{q^n}}

\newcommand{\cC}{{\mathcal C}}
\newcommand{\cS}{{\mathcal S}}

\newcommand{\cH}{{\mathcal H}}

\newcommand{\F}{{\mathbb F}}
\newcommand{\KK}{{\mathbb K}}

\newcommand{\K}{{\mathbb K}}
\newcommand{\fq}{{\mathbb F}_{q}}

\newcommand{\Stab}{\mathrm{Stab}}

\newcommand{\N}{\mathrm{N}}

\newcommand{\im}{\mathrm{Im}}
\newcommand{\rk}{\mathrm{rk}}

\newcommand{\Gal}{\mathrm{Gal}}
\newcommand{\GL}{\mathrm{GL}}
\newcommand{\End}{\mathrm{End}}
\newcommand{\Aut}{\mathrm{Aut}}
\newcommand{\LL}{\mathbb{L}}

\newcommand{\id}{\mathrm{id}}

\newcommand{\gcrd}{\mathrm{gcrd}}
\newcommand{\lclm}{\mathrm{lclm}}
\newcommand{\Ann}{\mathrm{Ann}}

\newcommand{\Snsk}{\mathcal{S}_{n,s,k}}
\newcommand{\Sk}{\mathcal{S}_{k}}

\newcommand{\Cen}{\mathrm{Cen}}

\newcommand{\EF}{{\mathbb E}_F}

\newcommand{\lid}{\mathcal{I}_{\ell}}
\newcommand{\rid}{\mathcal{I}_{r}}

\newcommand{\E}{\mathbb{E}}

\usepackage{enumitem}

\title{On the autotopism groups and the equivalence of finite cyclic semifields}

\author[P. Santonastaso]{Paolo Santonastaso}
\address{Paolo Santonastaso, \textnormal{Department of Mathematics and Applications ``R. Caccioppoli'', University of Naples Federico II, Via Cintia, Monte Sant'Angelo, 80126 Naples, Italy \newline
		Department of Mechanics, Mathematics and Management, Polytechnic University of Bari, 
		70125 Bari, Italy }}\email{paolo.santonastaso@unina.it}

\author[Y. Zhou]{Yue Zhou}
\address{Yue Zhou, \textnormal{College of Science, National University of Defense Technology, 410073 Changsha,
		China}}
\email{yue.zhou.ovgu@gmail.com}

\date{\today}

\begin{document}

	\begin{abstract}
		Special cases of finite cyclic semifields were first constructed by Hughes and Kleinfeld in 1960, and later by Sandler in 1962 and Knuth in 1965. The general construction of cyclic semifields was subsequently introduced by Petit in 1966, and later rediscovered from the perspective of irreducible semilinear transformations by Jha and Johnson in 1989. Since Sandler's foundational work in 1962, the complete determination of the autotopism groups of cyclic semifields and the full resolution of the isotopy problem for this family have remained long-standing open problems. The most significant advances in determining these autotopism groups are due to Dempwolff in 2011, who left open the case in which the field extension degree strictly divides the degree of the polynomial defining the semifield.
		
		In this paper, we provide a complete classification of cyclic semifields up to isotopy, together with the full determination of their autotopism groups, thereby closing the remaining cases left open by Dempwolff. Since cyclic semifields arise as a special instance of a broader family of maximum rank distance (MRD) codes constructed via skew polynomials, our methods also yield a complete classification of these MRD codes up to linear and semilinear equivalence over the prime field, together with an explicit description of their full automorphism groups.
	\end{abstract}
	
	\maketitle
	

	\section{Introduction}
	
	A \emph{(finite) semifield} is an algebraic structure satisfying all the axioms of a skewfield except, possibly, associativity. Equivalently, it is a finite division algebra that need not be associative. If the existence of a multiplicative identity is not required, one speaks of a \emph{presemifield}. Finite fields provide the associative examples, and Wedderburn's theorem shows that associativity in the finite case already forces commutativity. In this sense, finite semifields may be regarded as the natural nonassociative counterparts of finite fields.
	
	The theory begins with Dickson's 1906 construction of the first genuinely nonassociative finite division algebras \cite{dickson1906commutative}. It was subsequently developed by Albert, who introduced isotopy and the family of generalised twisted fields \cite{albert1961generalized}, and by Knuth, who clarified the connections between semifields, projective planes, and nonsingular tensors \cite{knuth1965finite}. Since then, finite semifields have occupied a central place in the interaction between algebra, finite geometry, and, more recently, coding theory. We refer to \cite{lavrauw2011finite,polverino2010linear,kantor2006finite} for general surveys on finite semifields and their connections.
	
	The depth and variety of the subject are reflected in a substantial body of work on the construction, classification, and isotopy of distinguished families of semifields. In this direction, one may point, for instance, to the classification results for semifield flocks obtained in \cite{blokhuis2003classification}, to the structural study of commutative presemifields and semifields in \cite{coulter2008commutative}, and to the two-parameter family introduced in \cite{zhou2013twoparameter}. More recently, G\"olo\u{g}lu and K\"olsch proved an exponential lower bound on the number of pairwise non-isotopic commutative semifields of odd order \cite{gologlu2023exponential}. For even order, this number is much bigger because of the very large family constructed by Kantor \cite{kantor_commutative_2003,kantor2004symplectic} which is a generalization of Knuth's binary semifields \cite{knuth_class_1965}. These results underscore both the richness of the isotopy problem and the difficulty of obtaining complete classification theorems even for highly structured subclasses.
	
	Of particular importance for the present paper is the construction arising from skew-polynomial rings and semilinear transformations. Ore \cite{ore1933theory} and Jacobson \cite{jacobson1937pseudo} studied associative algebras defined via skew-polynomial rings, as generalisations of cyclic algebras and Cayley--Dickson algebras. Let $\K$ be a field, $\sigma$ be an automorphism of $\K$ with fixed field $\F$, the skew polynomial ring 
	$$R:=\K[X;\sigma]=\{
	\alpha_rX^r+\alpha_{r-1}X^{r-1}+\cdots+\alpha_0~:~\alpha_i\in\K\}$$ 
	is defined with the ordinary  addition and the multiplication determined by the rule
	\[
	X \alpha=\sigma(\alpha)X
	\qquad \text{for every } \alpha\in\K.
	\]
	
	Petit later made an explicit construction of semifields using the skew-polynomial rings in \cite{petit1966certains}, while Jha and Johnson rediscovered the same class of objects for finite field $\K$ from the viewpoint of irreducible semilinear transformations \cite{Jha1989analog}. 
	Let \(f\in R\) be a monic irreducible polynomial of degree \(s\). Then the quotient \(R/Rf\) can be identified with the set of skew polynomials of degree at most \(s-1\), and it becomes a nonassociative algebra under the multiplication
	\[
	a\circ b:=ab \bmod_r f.
	\]
	Since \(R\) is a Euclidean domain and \(f\) is irreducible, this multiplication has no zero divisors. Therefore,
	\[
	S_f:=(R/Rf,+,\circ)
	\]
	is a semifield, called the \textbf{cyclic semifield} (or \textbf{Petit algebra}) associated with \(f\).

	One construction of semifields by Knuth \cite{knuth1965finite}, Hughes-Kleinfeld semifields \cite{hughes1960seminuclear} and Sandler's construction \cite{sandler1962autotopism} are all special cases of cyclic semifields; see \cite[Section 4]{brownpumplunsteele2018automorphisms}. It was shown in \cite{lavrauw2013semifields} that the cyclic, or Jha--Johnson, semifields introduced in \cite{Jha1989analog} are isotopic to Petit's semifields.
	
	The study of isotopism between distinct cyclic semifields, as well as their autotopism groups, also has a long history. To the best of our knowledge, it was initiated by Sandler in 1962. In \cite{sandler1962autotopism}, Sandler focused on the special case where $f=X^s-\delta \in \K[X]$ with $\K=\F_{q^n}$ and $\F=\F_q$, and he completely determined the autotopism groups of $R_f$. Up to now, the best results on these autotopism groups were obtained by Dempwolff in \cite{Dempwolff2011autotopism}, who determined them except for the case $n\mid s$ and $n<s$, in which $s=\deg(f)$.	
	Some further results can also be found in \cite{brownpumplun2018automorphisms,brownpumplunsteele2018automorphisms}. Meanwhile, the full solution to the isotopism problem for different cyclic semifields remains open, with only partial results available in \cite{kantor2008semifields,lavrauw2013semifields}.
	
	More recently, semifields have also acquired a prominent role in coding theory, since semifield spread sets give rise to special classes of maximum rank distance (MRD) codes. In particular, cyclic semifields were shown in \cite{sheekey2020new} to arise as special instances of a much broader skew-polynomial construction of MRD codes, encompassing most of the previously known algebraic families; this framework was subsequently extended further in \cite{lobillo2025quotients}, yielding new semifields and MRD codes for infinitely many choices of parameters. This connection places problems of isotopy and autotopism in direct correspondence with questions of equivalence and automorphisms in rank-metric coding theory. For instance, Sheekey raised an open question on the equivalence of these rank-metric codes in \cite[Remark~11]{sheekey2020new}.
	
	\subsection{Our contribution}
	The most important contribution of this paper is the resolution of the following two long-standing open problems:
	\begin{enumerate}[label=(\Roman*)]
		\item A complete classification of cyclic semifields up to isotopy, and, in parallel, the classification of the associated cyclic skew-polynomial MRD codes up to linear and semilinear equivalence.
		\item The full automorphism groups of these codes and, for the special case \(k=1\), the full autotopism groups of the corresponding cyclic semifields.
	\end{enumerate} 
	
	Our results are obtained by combining two complementary descriptions of the same family: the skew-polynomial model and the semilinear-transformation model. Let us consider the skew polynomial ring
	$
	R=\F_{q^n}[X;\sigma],$
	where \(\sigma\) generates \(\Gal(\F_{q^n}/\F_q)\), and let \(F(Y)\in\F_q[Y]\) be a monic irreducible polynomial of degree \(s\). We consider the quotient
	$
	R_F:=R/RF(X^n),$
	which is a simple Artinian ring; indeed,
	$
	R_F\cong M_n(\F_{q^s}).
	$
	For each integer \(1\le k\le n-1\), we consider the family introduced in \cite{sheekey2020new}
	\[
	\mathcal S_k(F):=
	\left\{
	\sum_{i=0}^{sk-1}\alpha_iX^i+RF(X^n)\;:\;\alpha_i\in\F_{q^n}
	\right\}
	\subseteq R_F.
	\]
	After identifying \(R_F\) with a matrix algebra over \(\F_{q^s}\), this is an MRD code in \(M_n(\F_{q^s})\).
	
	We give a description in terms of semilinear transformations of this family. If \(f\in R\) is an irreducible right divisor of \(F(X^n)\), then left multiplication by \(X\) induces a \(\sigma\)-semilinear map
	$
	L_X:R/Rf\rightarrow R/Rf.$
	This map is irreducible. Hence, after choosing a suitable \(\F_{q^n}\)-vector space \(V\) of dimension \(s\), it is conjugate to an irreducible \(\sigma\)-semilinear operator
	\[
	T_F\in \End_{\E}(V),
	\qquad \E\cong\F_{q^s}.
	\]
	Under the induced isomorphism
	$
	R_F\cong \End_{\E}(V),$
	the code \(\Sk(F)\) corresponds to the semilinear code
	\[
	\mathcal C_k(F):=
	\left\{
	\sum_{i=0}^{sk-1}\alpha_i\,T_F^i\;:\;\alpha_i\in\F_{q^n}
	\right\}
	\subseteq \End_{\E}(V).
	\]
	Accordingly, the family may be studied either inside a quotient of the skew-polynomial ring or as the span of powers of an irreducible semilinear operator. The first model is particularly well suited to the internal algebraic structure of the codes, whereas the second is the natural framework for equivalence and automorphism questions.
	
	Our first main result provides complete classifications of these codes up to linear equivalence and semilinear equivalence over the prime field, respectively.
	For $(\lambda,\tau) \in \F_q^\times \rtimes \Gal(\F_q/\F_p)$, we define its group action on $\mathcal I_{q,s}:=
	\{F(Y)\in\F_q[Y]: F \text{ monic irreducible of degree }s\}$ as follows:
	\begin{equation*}
		%		\label{eq:action.on.Iqs}
		(\lambda,\tau)\cdot F
		:=
		\lambda^sF^\tau(\lambda^{-1}Y),
		\qquad F\in\mathcal{I}_{q,s},
	\end{equation*}
	where $F^\tau(X) := \sum_{i} \tau(a_i)X^i$ for $F(X)=\sum_{i} a_i X^i$.
	
	Let $\iota$ stand for the reciprocal involution on $\mathcal{I}_{q,s}$, i.e.,
	\[
	\iota(F)=\widehat F \qquad \text{with } \widehat F(Y):=F(0)^{-1}Y^sF(Y^{-1}) \text{ for } F\in\mathcal{I}_{q,s}.
	\]
	For $n\geq 2$, define groups $\mathcal{G}$ and $\widehat{\mathcal{G}}$ by
	\[
	\mathcal{G} = 
	\begin{cases}
		\F_q^\times , & n>2,\\
		\F_q^\times \rtimes \langle \iota\rangle, & n=2,
	\end{cases}
	\]
	and 
	\[
	\widehat{\mathcal{G}} = \mathcal{G} \rtimes \Gal(\F_q/\F_p).
	\]
	We show that two MRD codes $\Sk(F)$ and $\Sk(G)$ defined from $R=\F_{q^n}[X;\sigma]$ are linear (semilinear, resp.) equivalent if and only if their associated irreducible polynomials $F$ and $G$ are in the same $\mathcal{G}$-orbit ($\widehat{\mathcal{G}}$-orbit, resp.) on $\mathcal{I}_{q,s}$.
	
	These classification theorems have immediate enumerative consequences. 
	If \(s>1\), \(1\le k\le n-1\), and \(n>2\), then the number of pairwise linearly inequivalent codes of the form \(\Sk(F)\) is at least
	\[
	\left\lceil \frac{|\mathcal{I}_{q,s}|}{\varepsilon_n(q-1)}\right\rceil,
	\]
	with 
	\[
	\varepsilon_n=\begin{cases}
		1, & n>2;\\
		2, & n=2.
	\end{cases}
	\]
	Likewise, if \(q=p^r\), then the number of pairwise \(\Gamma\)-inequivalent codes of the same form is at least
	\[
	\left\lceil \frac{|\mathcal{I}_{q,s}|}{\varepsilon_n r(q-1)}\right\rceil.
	\]
	In particular, the parameter \(F\) gives rise to a genuinely large supply of inequivalent MRD codes within the skew-polynomial family.
	
	In the semifield case \(k=1\), for \(q=p^r\), \(s>1\), and \(n>2\), we obtain at least
	\[
	\left\lceil \frac{|\mathcal{I}_{q,s}|}{\varepsilon_n r(q-1)}\right\rceil
	\]
	pairwise non-isotopic cyclic semifields of order \(q^{ns}\), with left and middle nuclei isomorphic to \(\F_{q^n}\) and right nucleus isomorphic to \(\F_{q^s}\). The strongest currently known counting results for odd-order semifields are confined to square orders, notably for the Taniguchi family \cite{gologlu2024counting} and for the biprojective commutative families of G\"olo\u{g}lu and K\"olsch \cite{gologlu2023exponential}. The cyclic semifields construction, by contrast, is available also in odd non-square orders. In this regime it yields the largest explicit family whose isotopy classes are completely determined and whose full autotopism groups can be described explicitly. Moreover, we also determine the automorphism groups of these objects excluding the very special case $(n,s,F)= (2,2,Y^2-\nu)$ with $\nu\in \F_q^\times$ (which implies $k=1$), because for these parameters the cyclic semifields $\cS_1(F)$ are isotopic to Hughes-Kleinfeld semifields whose autotopism groups have been already determined by Sandler in \cite{sandler1962autotopism}. 
	
	On the linear side, we prove that
	\[
	|\Aut_{\E_F}(\Sk(F))|
	=
	n\,\frac{(q^s-1)(q^n-1)^2}{q-1},
	\]
	and obtain an explicit description of this group. Passing to the semilinear model, we then determine the full semilinear automorphism group through a short exact sequence
	\[
	1\longrightarrow \Aut_{\E}(\mathcal C_k(F))
	\longrightarrow \Aut(\mathcal C_k(F))
	\longrightarrow \cH_F
	\longrightarrow 1,
	\]
	where \(\cH_F\) is the subgroup of \(\Aut_{\F_p}(\E)\). In particular,
	\[
	|\Aut(\mathcal C_k(F))|
	=
	|\cH_F|\,
	n\,\frac{(q^s-1)(q^n-1)^2}{q-1}.
	\]
	
	Finally, when \(k=1\), semilinear automorphisms of the spread set coincide with autotopisms of the associated cyclic semifield. We therefore obtain the full autotopism group $\Aut(\mathbb{S}_F)$ of the cyclic semifield $\mathbb{S}_F$
	which completes, in particular, the open cases in \cite{Dempwolff2011autotopism}.
	
	\subsection{Organization of this paper}
	The paper is organized as follows. Section~\ref{sec:preliminaries} collects the necessary background on semifields, MRD codes, their invariants and skew-polynomial rings. In Section \ref{sec:semilineartransf}, we establish the connections between skew-polynomial rings and irreducible semilinear operators. Section~\ref{sec:structuralresults} develops the structural properties of the family under consideration. In Section~\ref{sec:equivalence}, we present the classification results for linear and semilinear equivalence, together with their enumerative consequences. Finally, we determine the full automorphism group of the codes and, in the semifield case, the full autotopism group in Section \ref{sec:full-autotopism}.
	
	\section{Preliminaries on Semifields and Skew Polynomial Rings} \label{sec:preliminaries}
	\subsection{Semifields and rank-metric codes}
	Let \(\F\) be a finite field, and let \(\mathcal S\) be a finite-dimensional \(\F\)-vector space endowed with an \(\F\)-bilinear multiplication
	\[
	\star:\mathcal S\times \mathcal S \longrightarrow \mathcal S.
	\]
	Then \((\mathcal S,+,\star)\) is an \(\F\)-algebra, not assumed to be associative unless explicitly stated. We say that \((\mathcal S,+,\star)\) is a \textbf{presemifield} if every nonzero element induces bijective left and right multiplication maps, that is, if for each \(a\in \mathcal S\setminus \{0\}\) the maps
	\[
	L_a:b\mapsto a\star b,
	\qquad
	R_a:b\mapsto b\star a
	\]
	are permutations of \(\mathcal S\). If, in addition, \((\mathcal S,+,\star)\) has a multiplicative identity, then it is called a \textbf{semifield}. In other words, a finite semifield is a finite division algebra, not necessarily associative.
	
	Two presemifields \((\mathcal S,+,\star)\) and \((\mathcal S',+,\diamond)\) over the same prime field are said to be \textbf{isotopic} if there exist bijective additive maps
	$
	f,g,h:\mathcal S\longrightarrow \mathcal S'
	$
	such that
	\[
	f(a\star b)=g(a)\diamond h(b)
	\qquad\text{for all }a,b\in\mathcal S.
	\]
	If one may choose \(f=g=h\), then the two algebras are isomorphic. The notion of isotopy was introduced by Albert in \cite{albert1942non} as a generalization of isomorphism, and it is the natural equivalence relation in the theory of finite semifields, since isotopic semifields coordinatize isomorphic projective planes. An isotopy from \(\mathcal S\) to itself is called an \textbf{autotopism}. The set of all autotopisms of \(\mathcal S\) forms a group, called the \textbf{autotopism group} of \(\mathcal S\), and denoted by
	\[
	\Aut(\mathcal S)
	:=
	\bigl\{(f,g,h)\in \GL_{\F_p}(\mathcal S) \times \GL_{\F_p}(\mathcal S) \times \GL_{\F_p}(\mathcal S) :\,
	f(a\star b)=g(a)\star h(b)\ \text{for all }a,b\in\mathcal S
	\bigr\}.
	\]

	We note that properties such as the existence of a multiplicative identity or commutativity are not, in general, preserved under isotopy. On the other hand, every presemifield is isotopic to a semifield.
	
	We define the \textbf{spread set} of \(\mathcal S\) by
	\[
	\mathcal C(\mathcal S):=\{L_a : a\in \mathcal S\}\subseteq \End_{\F}(\mathcal S).
	\]
	Since \(\mathcal S\) is a semifield, every nonzero left multiplication map is invertible. Moreover, the assignment \(a\mapsto L_a\) is \(\F\)-linear, and hence \(\mathcal C(\mathcal S)\) is an \(\F\)-subspace of \(\End_{\F}(\mathcal S)\) of dimension \(\dim_{\F}(\mathcal S)\), all of whose nonzero elements are invertible. Conversely, let \(\mathcal C\subseteq \End_{\F}(\mathcal S)\) be an \(\F\)-subspace of dimension \(\dim_{\F}(\mathcal S)\) such that every nonzero element of \(\mathcal C\) is invertible. Fix an \(\F\)-linear isomorphism
	$
	\varphi:\mathcal S\longrightarrow \mathcal C$.
	Then \(\mathcal S\) acquires a semifield multiplication via
	\[
	x\star_{\varphi} y := \varphi(x)(y),
	\qquad x,y\in \mathcal S.
	\]
	Thus, finite semifields may be equivalently described by full-dimensional \(\F\)-subspaces of \(\End_{\F}(\mathcal S)\) consisting entirely of invertible maps, apart from the zero map. Different choices of \(\varphi\) yield isotopic semifields. In particular, spread sets provide special instances of MRD codes, as we now recall.
	
	A \textbf{rank-metric code} is a subset \(\mathcal C\subseteq M_n(\F)\), endowed with the metric
	\[
	d(A,B):=\rk(A-B).
	\]
	Its minimum distance is
	\[
	d(\mathcal C):=\min\{\rk(A-B):A,B\in\mathcal C,\ A\neq B\}.
	\]
	
	If \(\F'\) is a subfield of \(\F\), we say that \(\mathcal C\) is \textbf{\(\F'\)-linear} if \(\mathcal C\) is an \(\F'\)-vector subspace of \(M_n(\F)\). A fundamental constraint on the parameters of a rank-metric code is given by the Singleton bound of Delsarte \cite{delsarte1978bilinear}: if \(\mathcal C\subseteq M_n(\F)\) has minimum distance \(d\), then
	\begin{equation}\label{eq:singleton_rephrased}
		|\mathcal C|\leq |\F|^{\,n(n-d+1)}.
	\end{equation}
	Codes attaining equality in \eqref{eq:singleton_rephrased} are called \textbf{maximum rank-distance (MRD) codes}. Delsarte in \cite{delsarte1978bilinear} established the existence of MRD codes over arbitrary finite fields for all admissible parameters. More precisely, given a finite field \(\F\) and integers \(n,d\) such that \(d\leq n\), one can construct \(\F\)-linear MRD codes with these parameters. The same families were later independently rediscovered by Gabidulin in the equivalent vector representation setting; they are now usually called \emph{Gabidulin codes}, or \emph{Delsarte--Gabidulin codes}. In the special case \(n=d\), additive MRD codes correspond to semifields, as discussed above.
	
	Let
	\[
	\Gamma:=\GL_n(\F)\times \GL_n(\F)\times \Aut(\F),
	\]
	acting on \(M_n(\F)\) by
	\[
	(U,V,\rho)\cdot A:=UA^\rho V,
	\qquad A\in M_n(\F),
	\]
	where \(A^\rho\) is obtained by applying \(\rho\) entrywise to \(A\). This action extends naturally to subsets of \(M_n(\F)\). Two rank-metric codes are said to be \textbf{$\Gamma$-equivalent} if they lie in the same \(\Gamma\)-orbit.
	If \(\rho=\mathrm{id}\), then \(\mathcal C\) and \(\mathcal C'\) are said to be \textbf{linearly equivalent}. For a rank-metric code \(\mathcal C\subseteq M_n(\F)\), we define its \textbf{automorphism group} by
	\[
	\Aut(\mathcal C)
	:=
	\{(A,B,\rho)\in \Gamma : A\mathcal C^\rho B=\mathcal C\},
	\]
	and its \textbf{linear automorphism group} by
	\[
	\Aut_{\F}(\mathcal C)
	:=
	\{(A,B)\in \GL_n(\F)\times \GL_n(\F):A\mathcal C B=\mathcal C\}.
	\]
	%When \(\mathcal C=\mathcal C(\mathcal S)\) is the spread set of a semifield \(\mathcal S\), the group \(\Aut(\mathcal C)\) identifies with the \textbf{autotopy group} of \(\mathcal S\).
	
	\subsection{Nuclei, idealisers and equivalence of spread sets}\label{subsec:equivalencespreads}
	
	Let \((\mathcal S,+,\star)\) be a semifield. Its \textbf{left}, \textbf{middle}, and \textbf{right nucleus} are defined, respectively, by
	\begin{align*}
		\mathbb N_l(\mathcal S)
		&=\{a\in \mathcal S : a\star (b\star c)=(a\star b)\star c \text{ for all } b,c\in \mathcal S\},\\
		\mathbb N_m(\mathcal S)
		&=\{b\in \mathcal S : a\star (b\star c)=(a\star b)\star c \text{ for all } a,c\in \mathcal S\},\\
		\mathbb N_r(\mathcal S)
		&=\{c\in \mathcal S : a\star (b\star c)=(a\star b)\star c \text{ for all } a,b\in \mathcal S\}.
	\end{align*}
	Their intersection
	\[
	\mathbb N(\mathcal S)
	:=\mathbb N_l(\mathcal S)\cap \mathbb N_m(\mathcal S)\cap \mathbb N_r(\mathcal S)
	\]
	is the \emph{nucleus} of \(\mathcal S\), and the \emph{center} of \(\mathcal S\) is
	\[
	Z(\mathcal S)
	:=\{a\in \mathbb N(\mathcal S) : a\star b=b\star a \text{ for all } b\in \mathcal S\}.
	\]
	
	Each nucleus is a field. The center is the largest field over which \(\mathcal S\) is an algebra. Moreover, the nuclei and the center are fundamental isotopy invariants of a semifield.
	
	These notions admit a natural reformulation in the language of rank-metric codes, see \cite{lunardon2018nuclei,sheekey2020new}. Let \(\mathcal C\subseteq M_n(\F)\). The \textbf{left idealiser} and \textbf{right idealiser} of \(\mathcal C\) are
	\[
	\lid(\mathcal C):=\{A\in M_n(\F):A\mathcal C\subseteq \mathcal C\},
	\qquad
	\rid(\mathcal C):=\{B\in M_n(\F):\mathcal C B\subseteq \mathcal C\},
	\]
	respectively. The \textbf{centraliser} of \(\mathcal C\) is
	\[
	\Cen(\mathcal C):=\{A\in M_n(\F):AB=BA \text{ for all } B\in \mathcal C\},
	\]
	and its \textbf{center} is
	\[
	Z(\mathcal C):=\lid(\mathcal C)\cap \Cen(\mathcal C).
	\]
	
	The nuclei of a semifield admit a natural interpretation in terms of the associated spread set. Indeed, let
	$
	\mathcal C(\mathcal S)=\{L_a:a\in \mathcal S\}\subseteq \End_{\F}(\mathcal S),$ with $
	L_a(x)=a\star x,
	$
	be the spread set of \((\mathcal S,+,\star)\). Then the left and middle nuclei correspond, respectively, to the left and right idealisers of \(\mathcal C(\mathcal S)\), while the right nucleus corresponds to the centraliser of \(\mathcal C(\mathcal S)\). Likewise, the center of \(\mathcal S\) is naturally identified with the center of the spread set. Thus, the classical nuclei of a semifield may be recovered from the idealisers and the centraliser of its spread set. See \cite{marino2012nuclei,sheekey2020new}, for more details.
	
	For the purposes of this paper, it is convenient to realise the spread set over the right nucleus. Set
	\[
	\E:=\mathbb N_r(\mathcal S),
	\qquad
	m:=\dim_{\E}(\mathcal S).
	\]
	Then \(\mathcal S\) is naturally a right \(\E\)-vector space. Indeed, for every \(a\in \mathcal S\), the left multiplication map $L_a:b\mapsto a\star b$
	is \(\E\)-linear, since for all \(b\in \mathcal S\) and \(\lambda\in \E\) one has
	\[
	L_a(b\lambda)=a\star (b\lambda)=(a\star b)\lambda=L_a(b)\lambda.
	\]
	Hence, the spread set
	\[
	\mathcal C(\mathcal S)=\{L_a:a\in \mathcal S\}\subseteq \End_{\E}(\mathcal S)\cong M_m(\E).
	\]
	In particular, \(\mathcal C(\mathcal S)\) may be regarded as an additive MRD code in
	\(M_m(\E)\), with minimum distance \(m\). This right-nucleus representation is also the natural one for isotopy and autotopisms. Indeed, the following holds.
	
	\begin{proposition} [see \textnormal{\cite[Theorem 7]{lavrauw2011finite} and \cite[Proposition 3]{sheekey2020new}}] \label{prop:isotopy-equivalence-right-nucleus}
		Let \(\mathcal S\) and \(\mathcal S'\) be semifields with common right nucleus \(\E\), and let
		\[
		\mathcal C(\mathcal S),\mathcal C(\mathcal S')\subseteq \End_{\E}(\mathcal S)\cong M_m(\E)
		\]
		be their associated spread sets.
		Then \(\mathcal S\) and \(\mathcal S'\) are isotopic if and only if there exist
		\(A,B\in \GL_{\E}(\mathcal S)\) and \(\rho\in \Aut(\E)\) such that
		\[
		\mathcal C(\mathcal S')
		=
		A\,\mathcal C(\mathcal S)^{\rho}\,B.
		\]
		Here
		\[
		\mathcal C(\mathcal S)^{\rho}
		:=
		\{X^{\rho}:X\in \mathcal C(\mathcal S)\},
		\]
		where, after fixing an \(\E\)-basis, \(\rho\) acts entrywise on matrices.
	\end{proposition} In particular, taking
	\(\mathcal S'=\mathcal S\), the autotopism group of \(\mathcal S\) is identified with the automorphism group of its associated spread set. Thus, after choosing an \(\E\)-basis of
	\(\mathcal S\), we may write
	\begin{equation} \label{eq:reformulationautotopism}
		\Aut(\mathcal S)
		\cong
		\Aut(\mathcal C(\mathcal S))
		:=
		\left\{
		(A,B,\rho)\in \GL_m(\E)\times \GL_m(\E)\times \Aut(\E)
		:
		A\,\mathcal C(\mathcal S)^{\rho}\,B=\mathcal C(\mathcal S)
		\right\},
	\end{equation}
	and similarly for the subgroup of autotopisms which are linear over the right nucleus, we have
	\begin{equation} \label{eq:reformulationlinearautotopism}
		\Aut_{\E}(\mathcal S)
		\cong
		\Aut_{\E}(\mathcal C(\mathcal S))
		:=
		\left\{
		(A,B)\in \GL_m(\E)\times \GL_m(\E)
		:
		A\,\mathcal C(\mathcal S)\,B=\mathcal C(\mathcal S)
		\right\}.
	\end{equation}
	Accordingly, in the sequel we adopt this point of view.
	
	%	\section{Cyclic Semifields from Skew Polynomial Rings}\label{sec:cyclic.semi.skew.poly.ring}
	%	In this section, we provide a short introduction on cyclic semifields in the framework of skew polynomial rings.  
	
	\subsection{Skew polynomial rings} \label{sec:skewpolynomial}	
	We collect here the basic facts on skew polynomial rings that will be used throughout the paper. Standard references for this material include \cite{jacobson2009finite,goodearl2004introduction,ore1933theory}.
	
	Let
	$
	\sigma 
	$
	be a generator of the cyclic Galois group $\Gal(\F_{q^n}/\F_q)$. We consider the skew polynomial ring
	$
	R:=\F_{q^n}[X;\sigma],
	$
	whose underlying additive group consists of all formal polynomials
	\[
	\alpha_rX^r+\alpha_{r-1}X^{r-1}+\cdots+\alpha_0,
	\qquad \alpha_i\in\F_{q^n},
	\]
	with the usual addition and multiplication determined by the rule
	\[
	X \alpha=\sigma(\alpha)X
	\qquad \text{for every } \alpha\in\F_{q^n}.
	\]
	Equivalently, for monomials one has
	$
	(\alpha_iX^i)(\beta_jX^j)=\alpha_i\,\sigma^i(\beta_j)\,X^{i+j},
	$
	and multiplication is then extended by distributivity. For every nonzero $f\in R$, the degree $\deg(f)$ is defined exactly as in the commutative case. It satisfies the natural identities
	\[
	\deg(fg)=\deg(f)+\deg(g)
	\qquad\text{and}\qquad
	\deg(f+g)\leq \max\{\deg(f),\deg(g)\}
	\]
	for all nonzero $f,g\in R$.
	
	The ring $R$ is a left Euclidean domain. Thus, for any nonzero $f,g\in R$, there exist unique $q,r\in R$ such that
	$
	f=qg+r,
	$
	where either $r=0$ or $\deg(r)<\deg(g)$. If $r=0$, we say that $g$ right-divides $f$, and write
	$
	g\mid_r f.
	$
	In this case, $f$ is a left multiple of $g$. Accordingly, for nonzero $f_1,f_2\in R$ one may define their greatest common right divisor, denoted by
	$
	\gcrd(f_1,f_2),
	$
	and their least common left multiple, denoted by
	$
	\lclm(f_1,f_2).
	$
	If $a=\gcrd(f_1,f_2)$ and $b=\lclm(f_1,f_2)$, then
	\[
	Rf_1+Rf_2=Ra
	\qquad\text{and}\qquad
	Rf_1\cap Rf_2=Rb,
	\]
	where $Rf_i$ denotes the left ideal generated by $f_i$. As usual, a nonzero element $f\in R$ is called \textbf{reducible} if it admits a factorization
	$
	f=gh
	$
	with $\deg(g),\deg(h)>0$; otherwise $f$ is said to be \textbf{irreducible}. The ring $R$ is noncommutative as soon as $n>1$, and its center is given by
	$
	Z(R)=\F_q[X^n]$. 
	
	Now let $F(Y)\in\F_q[Y]$ be an irreducible polynomial of degree $s\geq 1$, with $F(Y)\neq Y$. Since $F(X^n)\in Z(R)$, the ideal $RF(X^n)$ is two-sided, and we may form the quotient ring
	$
	R_F:=R/RF(X^n)$.
	Every class in $R_F$ admits a unique representative of degree strictly smaller than $ns$, and therefore
	\[
	R_F=
	\left\{
	\alpha_0+\alpha_1X+\cdots+\alpha_{ns-1}X^{ns-1}+RF(X^n)
	:\ \alpha_0,\dots,\alpha_{ns-1}\in\F_{q^n}
	\right\}.
	\]
	
	For $\overline{a}\in R_F$, we always write
	$
	\overline{a}=a+RF(X^n),
	$
	where $a\in R$ is the unique representative with $\deg(a)<ns$. We then set
	$
	\deg(\overline{a}):=\deg(a)$. By \cite[Lemma 4.2]{gomez2019computing}, the center of $R_F$ is
	\[
	\mathbb{E}_F:=Z(R_F)\cong \frac{\F_q[Y]}{\F_q[Y]F(Y)}.
	\]
	In particular, $\mathbb{E}_F$ is a field of cardinality $q^s$, hence
	$\mathbb{E}_F\cong \F_{q^s}$. Also, since $F(Y)$ is irreducible then the two-sided ideal $RF(X^n)$ is maximal in $R$. As a consequence, the quotient $R_F$ is a central simple algebra over $\mathbb{E}_F$ of dimension $n^2$. By the Wedderburn--Artin theorem, there exists an $\EF$-algebra isomorphism
	\begin{equation}\label{eq:artin}
		\mathcal{M}_F:R_F\stackrel{\sim}{\longrightarrow} M_n(\EF).
	\end{equation}
	
	We will often identify $\overline{a}\in R_F$ with the matrix $\mathcal{M}_F(\overline{a})$. Note that, if
	$
	\mathcal{M}_F':R_F\stackrel{\sim}{\longrightarrow} M_n(\EF)
	$
	is another $\EF$-algebra isomorphism, then Skolem--Noether implies that there exists
	$
	N\in \GL_n(\EF)
	$
	such that
	$
	\mathcal{M}_F'(\overline{a})=N\mathcal{M}_F(\overline{a})N^{-1}$, for all $\overline{a}\in R_F$.
	Hence the rank of $\mathcal{M}_F(\overline{a})$ depends only on $\overline{a}$, and not on the chosen identification with $M_n(\EF)$. Accordingly, we write $\rk(\overline{a})$ for the rank of the associated matrix over $\EF$. The next result provides an explicit rank formula in terms of right divisors in the skew polynomial ring.
	
	\begin{theorem}\label{th:rankpolynomial_rephrased}
		Let $\overline{a}=a+RF(X^n)\in R_F$ be nonzero. Then
		\[
		\rk(\overline{a})
		=
		\frac{1}{s}\Bigl(\deg(F(X^n))-\deg\bigl(\gcrd(a,F(X^n))\bigr)\Bigr).
		\]
	\end{theorem}
	
	An element $g \in R$ is called \textbf{two-sided} if the left and right ideals it generates coincide, that is, $Rg = gR$.  
	By, for instance, \cite[Theorem~1.1.22]{jacobson2009finite}, an element $g \in R$ is two-sided if and only if it can be written in the form
	\[
	g = d G(X^n) X^m,
	\]
	for some $d \in \fqn$, a polynomial $G(Y) \in \F_q[Y]$, and an integer $m \ge 0$.
	
	\begin{definition}
		Let $f \in R$ be a nonzero element.  
		A \textbf{bound} of $f$ is a two-sided polynomial $f^* \in R$ such that the ideal $Rf^* = f^*R$ is the largest two-sided ideal contained in the left ideal $Rf$.  
		Equivalently,
		\[
		Rf^* = \Ann_R(R/Rf)
		:= \{\, g \in R \mid g(a + Rf) = 0 + Rf \text{ for all } a + Rf \in R/Rf \,\},
		\]
		where $\Ann_R(R/Rf)$ denotes the (left) annihilator of the $R$-module $R/Rf$.
	\end{definition}
	
	Thus, a bound of $f$ can be viewed as a minimal two-sided left multiple of $f$. If $f^* \neq 0$, the polynomial $f$ is said to be \textbf{bounded}.  
	Because $\sigma$ is an automorphism of $\fqn$ and $R$ has finite dimension $n^2$ over its center $Z(R) = \F_q[X^n]$, every nonzero element of $R$ is bounded, see e.g. \cite{gomez2019computing}.  Let $f \in R$ be a nonconstant polynomial with nonzero constant coefficient.  
	Then any bound $f^*$ of $f$ is of the form
	\[
	f^* = d F(X^n),
	\]
	for some $d \in \fqn$ and a monic polynomial $F(Y) \in \fq[Y]$ with nonzero constant term; see \cite[Lemma 2.11]{gomez2019computing}.  
	In this situation, we refer to the bound of $f$ as the unique monic central polynomial
	\[
	f^* = F(X^n).
	\]
	
	\subsection{Cyclic semifields and related MRD codes}
	
	Consider the skew polynomial ring
	$
	R=\F_{q^n}[X;\sigma],
	$
	where \(\sigma\in\Aut(\F_{q^n})\) has fixed field \(\F_q\). Let \(f\in R\) be a monic irreducible polynomial of degree \(s\). Then the quotient \(R/Rf\) can be identified with the set of skew polynomials of degree at most \(s-1\), and it becomes a nonassociative algebra under the multiplication
	\[
	a\circ b:=ab \bmod_r f.
	\]
	Since \(R\) is a Euclidean domain and \(f\) is irreducible, this multiplication has no zero divisors. Therefore,
	\[
	S_f:=(R/Rf,+,\circ)
	\]
	is a semifield, called the \textbf{cyclic semifield} (or \textbf{Petit algebra}) associated with \(f\). Moreover, the multiplication is associative if and only if \(Rf=fR\), equivalently, if and only if \(f\) generates a two-sided ideal; in particular, this happens whenever \(f\in Z(R)\). The nuclei of \(S_f\) are also explicitly known: if \(f\notin Z(R)\) and \(\deg(f)\geq 2\), then
	\[
	\mathbb N_l(S_f)\cong \F_{q^n},
	\qquad
	\mathbb N_m(S_f)\cong \F_{q^n},
	\qquad
	\mathbb N_r(S_f)\cong \F_{q^s},
	\qquad
	Z(S_f)\cong \F_q.
	\]
	
	In \cite{sheekey2020new}, this construction was extended to a much broader skew-polynomial framework for rank-metric codes. More precisely, once an irreducible polynomial \(F\in \F_q[Y]\) of degree \(s\) is fixed, as discussed in Section~\ref{sec:skewpolynomial}, the quotient
	$
	R_F:=R/RF(X^n)$
	inherits the structure of a central simple \(\E_F\)-algebra and may be identified with \(M_n(\E_F)\). Inside this algebra, one obtains a large family of MRD codes containing cyclic semifields as a distinguished special case.
	
	More precisely, let \(s\geq 1\). For each integer \(1\leq k\leq n-1\), set
	\begin{equation}\label{eq:skewGabidulin}
		\mathcal S_{n,s,k}(F):=
		\left\{
		\sum_{i=0}^{sk-1}\alpha_i X^i + RF(X^n)
		:\ \alpha_i\in \F_{q^n}
		\right\}
		\subseteq R_F\cong M_n(\E_F).
	\end{equation}
	Then \(\mathcal S_{n,s,k}(F)\) is an \(\F_q\)-linear MRD code in \(R_F\cong M_n(\E_F)\), of \(\F_q\)-dimension \(nsk\) and minimum distance \(n-k+1\).
	
	For brevity, once \(n\) and \(s\) are fixed, we write \(\mathcal S_k(F)\) in place of \(\mathcal S_{n,s,k}(F)\).

	We recall the basic invariants of the code \(\mathcal S_k(F)\), namely its left and right idealisers, its centraliser, and its center.
	
	\begin{proposition}[see \textnormal{\cite[Theorem~9]{sheekey2020new} and \cite[Theorem~5.18]{gomez2025adjoint}}] \label{prop:invariantsSNsk}
		Let \(s>1\), \(1\le k\le n-1\), and set
		$
		\mathcal C:=\mathcal S_k(F)$.
		Then
		\[
		\lid(\mathcal C)=\rid(\mathcal C)
		=
		\{\alpha+RF(X^n):\alpha\in \F_{q^n}\},
		\]
		\[
		\Cen(\mathcal C)=\E_F,
		\]
		and
		\[
		Z(\mathcal C)
		=
		\{\alpha+RF(X^n):\alpha\in \F_q\}.
		\]
	\end{proposition}
	We now explain how cyclic semifields arise within this construction. Let \(f\in R\) be a monic irreducible right divisor of \(F(X^n)\). Then \(\deg(f)=s\), and \(F(X^n)\) is the bound of \(f\). Set
	\[
	R_f:=R/Rf.
	\]
	Since \(F(X^n)\in Z(R)\), the quotient
	\[
	E_f:=\{z+Rf : z\in Z(R)\}
	\]
	is a field naturally isomorphic to
	$
	\EF=Z(R)/RF(X^n)\cong \F_{q^s}.
	$
	Hence \(R_f\) becomes naturally a right \(E_f\)-vector space, and, since \(|R_f|=q^{ns}\) and \(|E_f|=q^s\), it has dimension \(n\) over \(E_f\).
	
	For each \(\bar a=a+RF(X^n)\in R_F\), define
	\begin{equation}\label{eq:leftmultiplicationmaps}
		L_{\bar a}:R_f\longrightarrow R_f,
		\qquad
		b+Rf\longmapsto ab+Rf.
	\end{equation}
	This map is well defined and \(E_f\)-linear. In this way one obtains an \(E_f\)-algebra isomorphism
	\[
	\Lambda_f:R_F\longrightarrow \End_{E_f}(R_f),
	\qquad
	\bar a\longmapsto L_{\bar a}.
	\]
	After choosing an \(E_f\)-basis of \(R_f\), this yields
	\[
	R_F\cong \End_{E_f}(R_f)\cong M_n(E_f)\cong M_n(\F_{q^s}).
	\]
	
	Under this identification, the cyclic semifield \(S_f\) is recovered as the case \(k=1\): indeed, the spread set \(\C(S_f)\) corresponds precisely to the image of the elements of degree at most \(s-1\) in \(R_F\), namely to the code \(\mathcal S_{n,s,1}(F)\). Thus cyclic semifields appear as the first layer of the more general family of MRD codes arising from skew polynomial rings.
	
	\begin{remark}\label{rem:sheekey-general-family}
		The family \(\mathcal S_k(F)\) considered in this paper is exactly the \(\eta=0\)
		specialization of the more general family \(\Snsk(\eta,\rho,F)\) introduced in
		\cite{sheekey2020new}. It is worth recalling that the family \(\mathcal S_{n,s,k}(\eta,\rho,F)\) contains several
		classical constructions as special cases. In particular, for \(k=1\) one obtains semifields,
		including cyclic semifields and generalised twisted fields \cite{albert1961generalized}. Whereas for \(s=1\) one recovers Gabidulin codes
		\cite{delsarte1978bilinear,gabidulin1985theory,Gabidulins} and twisted Gabidulin codes \cite{sheekey2016new,lunardon2018generalized}.
		More generally, quotients of skew polynomial rings have also been studied beyond the finite-field
		matrix setting, leading to new division algebras and MRD codes over matrix spaces \(M_n(\mathbb D)\), where $\mathbb{D}$ is a non commutative division ring;
		see, for instance, \cite{sheekey2020new,thompson2023division,lobillo2025quotients}.
	\end{remark}

	By definition, for $n=1$ the skew polynomial ring $R$ is the usual commutative polynomial ring $\F_{q}[X]$. For $s=1$, $S_f$ becomes a finite field and $\cS_{n,1,k}(F)$ becomes the classical Gabidulin code. Therefore, \textbf{throughout the rest part of this paper, we always assume that \boldmath{$n\geq 2$ and $s\geq 2$}}.

	\section{Irreducible Semilinear Transformations and the Skew-Polynomial framework} \label{sec:semilineartransf} \label{subsection:constructionirreducible}
	
	In this section, we connect the skew‑polynomial setting with the framework of irreducible semilinear transformations. This link is crucial in the sequel for the equivalence problem and for the determination of the full automorphism group of cyclic semifields, as it allows us to switch between these two frameworks to leverage their respective advantages.
	
	Let \(\sigma\) be a generator of \(\Gal(\F_{q^n}/\F_q)\), and let \(V\) be an \(s\)-dimensional \(\F_{q^n}\)-vector space. A map \(T:V\to V\) is called \(\sigma\)-\textbf{semilinear} if \(T\) is additive and
	\[
	T(\alpha v)=\sigma(\alpha)T(v)
	\]
	for every \(\alpha\in \F_{q^n}\) and \(v\in V\).  A \(\sigma\)-semilinear map \(T\) is said to be \textbf{irreducible} if the only \(T\)-invariant \(\F_{q^n}\)-subspaces of \(V\) are \(\{0\}\) and \(V\). In \cite{Dempwolff2010onirreducible}, Dempwolff characterized irreducible semilinear transformations and provided an explicit description of them. We now recall this construction and we fix the following notation.
	
	Let \(n,s\) be positive integers, and set
	$
	\ell:=\gcd(n,s)$.
	We also define
	\begin{itemize}
		\item \(n':=\frac{n}{\ell}\),
		\item \(s':=\frac{s}{\ell}\),
		\item \(\KK:=\F_{q^n}\),
		\item \(\LL:=\F_{q^{s'n}}=\F_{q^{ns/\ell}}\),
		\item \(\F:=\F_{q^s}\).
	\end{itemize}
	
	Throughout this section, there will be several similar notations which stand for some algebraic structures all isomorphic to $\F_{q^s}$. To help the reader to follow their definitions, we list them here:
	\begin{itemize}
		\item $	E_f:=\{z+Rf : z\in Z(R)\}$;
		\item $\EF=Z(R)/RF(X^n)$;
		\item $\E_T:=\F_q[T^n]\subseteq \End_{\F_q}(V)$.
	\end{itemize}
	Here $Z(R)$ is the center of $R=\F_{q^n}[X;\sigma]$.
	
	Consider the vector space
	\[
	V:=\LL^\ell.
	\]
	Then \(V\) is an \(\ell\)-dimensional vector space over \(\LL\). Let
	$
	\sigma:\alpha\longmapsto \alpha^{q^h}$
	be a generator of \(\Gal(\F_{q^n}/\F_q)\), where \(\gcd(h,n)=1\). Since \(\ell=\gcd(n,s)\), there exist integers \(c,c'\) such that
	\[
	\ell=cs+c'n.
	\]
	We may choose \(c>0\) with \(\gcd(c,n')=1\). Define
	\[
	\gamma:\LL\longrightarrow \LL,
	\qquad
	\alpha\longmapsto \alpha^{\sigma^{cs}}=\alpha^{q^{hcs}}.
	\]
	For every \(\alpha\in \KK\), we have
	\[
	\gamma(\alpha)=\alpha^{\sigma^{cs}}=\alpha^{\sigma^\ell},
	\]
	since \(cs\equiv \ell \pmod n\). Moreover, \(\gamma\) fixes \(\F=\F_{q^s}\). Since \([\LL:\F]=n'\) and \(\gcd(ch,n')=1\), the automorphism \(\gamma\) is a generator of \(\Gal(\LL/\F)\).
	
	We now endow \(V\) with a \(\KK\)-vector space structure, defined as 
	\begin{equation}\label{eq:vectorspaceconstruction}
		\alpha \cdot (x_0,\ldots,x_{\ell-1})
		:=
		\bigl(\alpha x_0,\alpha^\sigma x_1,\ldots,\alpha^{\sigma^{\ell-1}}x_{\ell-1}\bigr),
	\end{equation}
	for every \(\alpha\in \KK\) and \((x_0,\ldots,x_{\ell-1})\in V\). In particular,
	\[
	\dim_{\KK}(V)=\ell[\LL:\KK]=\ell s'=s.
	\]

	For \(w\in \LL^\times\), define
	\[
	T_{w,\sigma}:V\longrightarrow V,
	\qquad
	(x_0,\ldots,x_{\ell-1})\longmapsto (x_1,\ldots,x_{\ell-1},w x_0^\gamma).
	\]
	For every \(\alpha\in \KK\), we have
	\[
	\begin{aligned}
		T_{w,\sigma}\bigl(\alpha\cdot(x_0,\ldots,x_{\ell-1})\bigr)
		&=
		(\alpha^\sigma x_1,\ldots,\alpha^{\sigma^{\ell-1}}x_{\ell-1},w\,\alpha^\gamma x_0^\gamma)\\
		&=
		(\alpha^\sigma x_1,\ldots,\alpha^{\sigma^{\ell-1}}x_{\ell-1},w\,\alpha^{\sigma^\ell}x_0^\gamma)\\
		&=
		\sigma(\alpha)\cdot T_{w,\sigma}(x_0,\ldots,x_{\ell-1}),
	\end{aligned}
	\]
	since \(\gamma\) restricts to \(\sigma^\ell\) on \(\KK\). Therefore \(T_{w,\sigma}\) is a \(\sigma\)-semilinear operator on \(V\), viewed as a $\K$-vector space.
	
	A direct computation gives
	\[
	T_{w,\sigma}^{\ell}(x_0,\ldots,x_{\ell-1})
	=
	(w x_0^\gamma,w x_1^\gamma,\ldots,w x_{\ell-1}^\gamma),
	\]
	and iterating this identity \(n'\) times, and using that \(\gamma\) generates
	\(\Gal(\LL/\F)\), hence has order \(n'=[\LL:\F]\), we obtain
	\begin{equation}\label{eq:powernirreducible}
		T_{w,\sigma}^{n}(x_0,\ldots,x_{\ell-1})
		=
		(u x_0,u x_1,\ldots,u x_{\ell-1}),
	\end{equation}
	where
	\[
	u=w\,w^\gamma\cdots w^{\gamma^{n'-1}}=\N_{\LL/\F}(w).
	\]
	
	\begin{proposition}[cf.\ \textnormal{\cite[Section~2.8]{Dempwolff2010onirreducible}}]
		Let \(w\in \LL^\times\) be such that $\F_q\bigl(\N_{\LL/\F}(w)\bigr)=\F$. Then the operator \(T_{w,\sigma}\) is irreducible.
	\end{proposition}
	
	The following theorem shows that every irreducible \(\sigma\)-semilinear operator
	on an \(s\)-dimensional \(\K\)-vector space is conjugate to one of the maps \(T_{w,\sigma}\).

	\begin{theorem}[cf.\ \textnormal{\cite[Theorem~3.4]{Dempwolff2010onirreducible}}]\label{th:classificationsirreducible}
		Let \(V\) be an \(s\)-dimensional \(\F_{q^n}\)-vector space. Let \(\sigma\) be a generator of \(\Gal(\F_{q^n}/\F_q)\). Then the following holds:
		\begin{enumerate}[label=(\alph*)]
			\item For any irreducible $\sigma$-semilinear map $T$ on $V$, there exist \(w\in \LL^\times\), with \(\fq\left(\N_{\LL/\F}(w)\right)=\F_{q^s}\), and an invertible \(\F_{q^n}\)-linear transformation $\phi$	such that
			$
			\phi^{-1}\circ T\circ \phi=T_{w,\sigma}.
			$
			\item The following are equivalent:
			\begin{enumerate}[label=(\roman*)]
				\item The irreducible $\sigma$-semilinear maps $T_{w,\sigma}$ and $T_{w',\sigma}$ are conjugate under $\GL_{\F_{q^n}}(V)$.
				\item The linear maps $T^n_{w,\sigma}$ and $T^n_{w',\sigma}$ are conjugate under $\GL_{\F_{q^n}}(V)$.
				\item $\N_{\LL/\F}(w)$ and $\N_{\LL/\F}(w')$ are conjugate under $\Gal(\F_{q^n}/\F_q)$. 
			\end{enumerate}
		\end{enumerate}
	\end{theorem}
	
%	\begin{theorem}[cf.\ \textnormal{\cite[Theorem~3.4]{Dempwolff2010onirreducible}}]
%		Let \(V'\) be an \(s\)-dimensional \(\F_{q^n}\)-vector space, and let \(T':V'\to V'\) be an irreducible \(\sigma\)-semilinear map, where \(\sigma\) is a generator of \(\Gal(\F_{q^n}/\F_q)\). Then there exist \(w\in \LL^\times\), with \(\fq\left(\N_{\LL/\F}(w)\right)=\F_{q^s}\), and an invertible \(\F_{q^n}\)-linear map
%		\[
%		\phi:V\longrightarrow V'
%		\]
%		such that
%		\[
%		\phi^{-1}\circ T'\circ \phi=T_{w,\sigma}.
%		\]
%	\end{theorem}	
	We now relate irreducible semilinear transformations to skew-polynomial modules.
	
	Let \(f\in R\) be an irreducible skew polynomial of degree \(s\), and let
	\[
	f^*=F(X^n),
	\qquad F(Y)\in \F_q[Y],
	\]
	be its bound. Consider the quotient
	\[
	R_f:=R/Rf,
	\]
	viewed as a left \(\F_{q^n}\)-vector space. Then \(R_f\) has dimension \(s\) over \(\F_{q^n}\).
	As recalled in \eqref{eq:leftmultiplicationmaps}, every class
	\(\bar a=a+RF(X^n)\in R_F\) induces an \(E_f\)-linear endomorphism \(L_{\bar a}\) of \(R_f\).
	In particular, we write \(L_X\) for the map induced by left multiplication by \(X+RF(X^n)\).
	
	When \(R_f\) is regarded as an \(\F_{q^n}\)-vector space, the map \(L_X\) is \(\sigma\)-semilinear. Indeed, for every \(\alpha\in \F_{q^n}\) and \(v\in R_f\),
	\[
	L_X(\alpha v)
	=X(\alpha v)\bmod_r f
	=\sigma(\alpha)\,(Xv\bmod_r f)
	=\sigma(\alpha)L_X(v).
	\]
	\begin{theorem}[see \textnormal{\cite[Theorem~13]{lavrauw2013semifields}}]\label{th:irreducibleLx}
		Let \(f\in R\) be an irreducible skew polynomial of degree \(s\). Then \(L_X\) is an irreducible \(\sigma\)-semilinear map on the \(\F_{q^n}\)-vector space \(R_f\).
	\end{theorem}
	
	The operator
	$
	L_{X^n}=(L_X)^n
	$
	is \(\F_{q^n}\)-linear on \(R_f\). Indeed, it is induced by the map
	$
	v\longmapsto X^n v \bmod_r f.
	$
	Its minimal polynomial is determined by the bound of \(f\).
	
	\begin{theorem}[see \textnormal{\cite[Theorem~3]{sheekey2020new}}]\label{th:minimalxn}
		Let \(f\neq X\) be an irreducible element of \(R\) of degree \(s\). Then the minimal polynomial of \(L_{X^n}\) on \(R_f\) over \(\F_q\) is equal to the bound of \(f\).
	\end{theorem}
	
	Combining Theorems~\ref{th:classificationsirreducible}, \ref{th:irreducibleLx}, and \ref{th:minimalxn}, we obtain the following description.
	
	\begin{theorem}\label{th:irreducibleroot}
		Let \(f\neq X\) be an irreducible element of \(R\) of degree \(s\), and let
		\[
		f^*=F(X^n),
		\qquad F(Y)\in \F_q[Y],
		\]
		be its bound. Then there exist \(w\in \LL^\times\) and an invertible \(\F_{q^n}\)-linear map
		\[
		\phi:R_f\longrightarrow V
		\]
		such that
		\[
		\phi\circ L_X\circ \phi^{-1}=T_{w,\sigma},
		\]
		where \(u:=\N_{\LL/\F}(w)\) is a root of \(F(Y)\).
	\end{theorem}
	
	\begin{proof}
		Since \(f\) is irreducible, Theorem~\ref{th:irreducibleLx} shows that \(L_X\) is an irreducible \(\sigma\)-semilinear map on the \(\F_{q^n}\)-vector space \(R_f\). Hence, by Theorem~\ref{th:classificationsirreducible}, there exist \(w\in \LL^\times\) and an invertible \(\F_{q^n}\)-linear map
		\[
		\phi:R_f\longrightarrow V
		\]
		such that
		\[
		\phi\circ L_X\circ \phi^{-1}=T_{w,\sigma}.
		\]
		Set
		\[
		u:=\N_{\LL/\F}(w).
		\]
		By \eqref{eq:powernirreducible}, we have
		\[
		T_{w,\sigma}^n=u\,\mathrm{id}_V.
		\]
		Therefore the minimal polynomial of \(T_{w,\sigma}^n\) over \(\F_q\) is the minimal polynomial of \(u\) over \(\F_q\). Since \(L_{X^n}=(L_X)^n\) and \(T_{w,\sigma}^n\) are conjugate, they have the same minimal polynomial over \(\F_q\). By Theorem~\ref{th:minimalxn}, the minimal polynomial of \(L_{X^n}\) is \(F(Y)\). Hence \(F(Y)\) is also the minimal polynomial of \(u\), and in particular
		\[
		F(u)=0.
		\]
		Thus \(u=\N_{\LL/\F}(w)\) is a root of \(F(Y)\), as claimed.
	\end{proof}
	
	We now make explicit the field structure naturally induced on $V$ by an irreducible
	$\sigma$-semilinear operator. Let $T\colon V\to V$ be an irreducible $\sigma$-semilinear map, where $V$ is an
	$s$-dimensional $\F_{q^n}$-vector space. Since $T^n$ is $\F_{q^n}$-linear, we may consider
	the $\F_q$-subalgebra
	\[
	\E_T:=\F_q[T^n]\subseteq \End_{\F_q}(V).
	\]
	By irreducibility, the minimal polynomial of $T^n$ over $\F_q$ has degree $s$.
	Hence $\E_T$ is a field of cardinality $q^s$, and therefore
	\[
	\E_T\cong \F_{q^s}.
	\]
	We regard $V$ as a left $\E_T$-vector space by setting
	\[
	\left(\sum_{i=0}^m \alpha_i(T^n)^i\right)\cdot v
	:=
	\sum_{i=0}^m \alpha_i\,T^{ni}(v),
	\qquad \alpha_i\in \F_q,\ v\in V.
	\]
	Since every element of $\E_T$ is a polynomial in $T^n$, it follows immediately that
	$T$ is $\E_T$-linear.
	
	We now use the conjugacy established in Theorem \ref{th:irreducibleroot} to identify \(R_F\) with \(\End_{\E_{T}}(V)\), where $T=T_{w,\sigma}$. 
	
	\begin{theorem}\label{th:polydescription}
		Let \(F(Y)\neq Y\) be an irreducible polynomial in \(\F_q[Y]\) of degree \(s\). Let \(w\in \LL^\times\) be such that
		\[
		u:=\N_{\LL/\F}(w)
		\]
		is a root of \(F(Y)\). Then the map
		\begin{equation} \label{eq:fromskewtosemilinear}
			\begin{array}{rccc}
				\Psi_w: & R_F & \longrightarrow & \End_{\E_{T_{w,\sigma}}}(V)\\[4pt]
				& \displaystyle
				\sum_{i=0}^{ns-1}\alpha_iX^i+RF(X^n)
				& \longmapsto &
				\displaystyle \sum_{i=0}^{ns-1}\alpha_i \cdot T_{w,\sigma}^i
			\end{array}
		\end{equation}
		is a ring isomorphism between $R_F$ and $\End_{\E_{T_{w,\sigma}}}(V)$.  Moreover, the map $\Psi_w$ is also $\K$-linear and $\Psi_w(\E_F)=\E_{T_{w,\sigma}}$.
	\end{theorem}
	
	\begin{proof} Let \(f\in R\) be a monic irreducible factor of \(F(X^n)\).  By \Cref{th:irreducibleLx}, we know that $L_X$ is conjugate, under an invertible map $\phi \in \mathrm{Hom}_{\F_{q^n}}(R_f,V)$, to $T_{w',\sigma}$, i.e. \[\phi\circ L_X \circ \phi^{-1} = T_{w',\sigma}.\] Here, $w' \in \LL$ is such that $u' = \N_{\LL/\F}(w')$ is a root of $F(Y)$. Now, $u = \N_{\LL/\F}(w)$ is also a root of $F(Y)$. Thus, by Theorem~\ref{th:classificationsirreducible}~(b), $T_{w,\sigma}$ and $T_{w',\sigma}$ are conjugate under an $\F_{q^n}$-linear map $\phi':V\rightarrow V$, i.e., \[\phi' \circ T_{w',\sigma} \circ \phi'^{-1} = T_{w,\sigma}.\]
		Hence, we have:
		\[
		(\phi' \circ \phi) \circ L_X \circ (\phi' \circ \phi)^{-1}=  \phi' \circ (\phi \circ L_X \circ \phi^{-1}) \circ \phi'^{-1}=  \phi' \circ T_{w',\sigma} \circ \phi'^{-1}=T_{w,\sigma}.
		\]
		Therefore, $L_X$ is conjugate to $T_{w,\sigma}$ through the $\F_{q^n}$-linear map $\phi' \circ \phi$. 
		Note also that
		\[
		(\phi' \circ \phi) \circ L_{\alpha X^i} \circ (\phi' \circ \phi)^{-1} = \alpha \cdot T_{w,\sigma}^i, 
		\]
		for any $\alpha \in \F_{q^n}$ and $i \leq sn-1$.
		Consequently, for any element $\overline{a}=\sum_{i=0}^{ns-1}\alpha_iX^i+RF(X^n)$, we have
		\[
		\begin{array}{rl}
			(\phi' \circ \phi) \circ L_{\overline{a}} \circ (\phi' \circ \phi)^{-1} & =
			(\phi' \circ \phi) \circ \left( \sum\limits_{i=0}^{ns-1}\alpha_iX^i+RF(X^n) \right) \circ (\phi' \circ \phi)^{-1} \\
			& = \sum\limits_{i=0}^{ns-1}\alpha_i\cdot T_{w,\sigma}^i \\
			&= \Psi_w(\overline{a}).
		\end{array}
		\] Moreover, since $\phi' \circ \phi$ is an invertible $\F_{q^n}$-linear map, we get that $\Psi_w$  is a ring isomorphism between $R_F$ and $\End_{\E_{T_{w,\sigma}}}(V)$. The \(\K\)-linearity is immediate from the definition. Finally, the center \(\E_F\) is generated by
		\(X^n+RF(X^n)\), and
		\[
		\Psi_w(X^n+RF(X^n))=T_{w,\sigma}^n=u\,\id_V.
		\]
		Therefore
		\[
		\Psi_w(\E_F)=\F_q[u]\,\id_V=\E_{T_{w,\sigma}}.
		\]
		This completes the proof.
	\end{proof}
	
	\section{Structural results on \texorpdfstring{$R_F$ and $\mathcal S_k(F)$}{R_F and S_{n,s,k}(F)}} \label{sec:structuralresults}
	
	The aim of this section is to derive the structural properties of \(R_F\) and of the associated code \(\mathcal S_k(F)\) that will be needed later in the study of equivalence and automorphisms. 
	
	We start by determining the centralizer of \(\F_{q^n}\) in \(R_F\), namely
	\[
	\Cen_{R_F}(\F_{q^n}):=\{\overline{g} \in R_F:\F_{q^n}\overline{g}=\overline{g}\F_{q^n}\}.
	\]
	
	\begin{proposition} \label{prop:descriptioncentraliser}
		The centralizer of \(\F_{q^n}\) in \(R_F\) is given by
		\[
		\Cen_{R_F}(\F_{q^n})
		=
		\left\{
		z(X^n)X^t+RF(X^n)
		:\ z(Y)\in\F_{q^n}[Y], \ t\in\{0,\dots,n-1\}
		\right\}.
		\]
	\end{proposition}
	
	\begin{proof}
		Let
		$
		\bar g=z(X^n)X^t+RF(X^n),
		$
		where \(z(Y)\in \F_{q^n}[Y]\), and \(t\in\{0,\dots,n-1\}\). We show that
		$
		\bar g\in \Cen_{R_F}(\F_{q^n})$.
		Let \(\alpha\in \F_{q^n}\). Since \(X^n\) commutes with every element of \(\F_{q^n}\), we have
		\[
		\alpha z(X^n)=z(X^n)\alpha.
		\]
		Moreover,
		\[
		\alpha X^t = X^t\sigma^{-t}(\alpha).
		\]
		Therefore
		\[
		\alpha z(X^n)X^t
		=
		z(X^n)\alpha X^t
		=
		z(X^n)X^t\sigma^{-t}(\alpha).
		\]
		Since \(\sigma^{-t}(\alpha)\in \F_{q^n}\), this shows that
		\[
		\alpha \bar g=\bar g\,\beta
		\qquad\text{with}\qquad
		\beta=\sigma^{-t}(\alpha)\in \F_{q^n}.
		\]
		Hence \(\F_{q^n}\bar g\subseteq \bar g\F_{q^n}\). Replacing \(\alpha\) by \(\sigma^t(\alpha)\), one also gets
		$
		\bar g\F_{q^n}\subseteq \F_{q^n}\bar g$. Thus
		$
		\F_{q^n}\bar g=\bar g\F_{q^n},$
		and so \(\bar g\in \Cen_{R_F}(\F_{q^n})\). For the reverse inclusion, let \(\overline{g}=g+RF(X^n) \in  \Cen_{R_F}(\fqn)\). Then, for every $\alpha \in \fqn$, there exist \(\beta\in\F_{q^n}\) such that
		\[
		\alpha \overline{g}=\overline{g}\beta.
		\]
		Write
		\[
		g=\sum_{i=0}^{ns-1} \gamma_iX^i,\qquad \gamma_i\in\F_{q^n}.
		\]
		Since \(X^i\beta=\sigma^i(\beta)X^i\) and $\deg(\alpha \overline{g})=\deg(\overline{g}\beta)<ns$ we obtain
		\[
		\alpha\sum_{i=0}^{ns-1}\gamma_iX^i
		=
		\sum_{i=0}^{ns-1}\gamma_iX^i\beta
		=
		\sum_{i=0}^{ns-1}\gamma_i\sigma^i(\beta)X^i.
		\]
		By comparing coefficients, it follows that
		$
		\alpha \gamma_i=\gamma_i\sigma^i(\beta),
		$
		for $i=0,\ldots,ns-1$. Equivalently,
		$
		\gamma_i(\alpha-\sigma^i(\beta))=0.
		$
		Hence, for each \(i\), either \(\gamma_i=0\) or \(\alpha=\sigma^i(\beta)\). Assume now that \(\gamma_i\neq 0\) and \(\gamma_j\neq 0\) for some \(i\neq j\). Then
		$
		\alpha=\sigma^i(\beta)=\sigma^j(\beta),
		$
		and therefore
		$
		\sigma^{i-j}(\beta)=\beta$.
		It follows that \(\beta\in \F_{q^{\gcd(n,i-j)}}\), and hence also \(\alpha\in \F_{q^{\gcd(n,i-j)}}\). Since this holds for any $\alpha \in \fqn$, this forces
		$
		n\mid (i-j)$.
		Therefore all nonzero coefficients of \(g\) occur in degrees belonging to a single residue class modulo \(n\). In other words, there exists \(t\in\{0,\dots,n-1\}\) such that
		\[
		g=\sum_{w=0}^{s-1} \gamma_{nw+t}X^{nw+t}=z(X^n)X^t,
		\]
		for some $z(Y) \in \fqn[Y]$.
	\end{proof}
	
	Proposition~\ref{prop:descriptioncentraliser} can be translated into the
	semilinear setting, as stated in the following proposition.
	
	\begin{proposition}\label{prop:semilinear-centralizer}
		Let \(F(Y)\neq Y\) be an irreducible polynomial in \(\F_q[Y]\) of degree \(s\). Let
		\(w\in \LL^\times\), set
		\[
		T:=T_{w,\sigma},
		\qquad
		u:=\N_{\LL/\F}(w),
		\]
		and assume that \(F(u)=0\). Define
		\[
		\mathcal Z:=\{z(u)\,\id_V : z(Y)\in \K[Y]\}\subseteq \End_{\E_T}(V).
		\]
		Then
		\[
		\Cen_{\End_{\E_T}(V)}(\K)
		=
		\left\{
		z(u)\,T^t : z(Y)\in \K[Y],\ t\in\{0,\dots,n-1\}
		\right\}.
		\]
		Moreover,
		\[
		\End_{\E_T}(V)=\bigoplus_{t=0}^{n-1}\mathcal Z\,T^t,
		\]
		and \(\mathcal Z\,T^t\) is precisely the space of \(\sigma^t\)-semilinear endomorphisms of \(V\)
		with respect to the \(\K\)-vector space structure \eqref{eq:vectorspaceconstruction}.
	\end{proposition}
	
	\begin{proof}
		By Proposition~\ref{prop:descriptioncentraliser}, in the quotient algebra \(R_F=R/RF(X^n)\) one has
		\[
		\Cen_{R_F}(\K)
		=
		\left\{
		z(X^n)X^t+RF(X^n): z(Y)\in \K[Y],\ t\in\{0,\dots,n-1\}
		\right\}.
		\]
		Applying the isomorphism \(\Psi_w\) of Theorem~\ref{th:polydescription}, and using
		\[
		\Psi_w(X+RF(X^n))=T
		\qquad\text{and}\qquad
		\Psi_w(X^n+RF(X^n))=T^n=u\,\id_V,
		\]
		we obtain
		\[
		\Psi_w\bigl(z(X^n)X^t+RF(X^n)\bigr)=z(u)\,T^t.
		\]
		This proves the description of the centralizer.
		
		Since each element of \(R_F\) has a unique representative of degree \(<ns\), every class in \(R_F\)
		admits a unique decomposition of the form
		\[
		\sum_{t=0}^{n-1} z_t(X^n)X^t+RF(X^n),
		\qquad z_t(Y)\in \K[Y].
		\]
		Applying \(\Psi_w\), we obtain
		\[
		\End_{\E_T}(V)=\bigoplus_{t=0}^{n-1}\mathcal Z\,T^t.
		\]
		
		Finally, every element of \(\mathcal Z\) is \(\K\)-linear, whereas \(T^t\) is
		\(\sigma^t\)-semilinear. Hence each element of \(\mathcal Z\,T^t\) is \(\sigma^t\)-semilinear.
		Conversely, the direct sum decomposition above shows that every endomorphism in \(\End_{\E_T}(V)\)
		has a unique decomposition according to its semilinear type. This concludes the proof.
	\end{proof}
	
	We shall repeatedly use the fact that
	$
	\Cen_{R_F}(\F_{q^n})^\times$
	is closed under inversion. Indeed, if \(\bar g\in \Cen_{R_F}(\F_{q^n})^\times\), then for every \(\lambda\in \F_{q^n}\) there exists \(\mu\in \F_{q^n}\) such that
	$
	\lambda \bar g=\bar g\mu$.
	Multiplying by \(\bar g^{-1}\) on the left and on the right, we obtain
	\[
	\bar g^{-1}\lambda=\mu \bar g^{-1}\in \F_{q^n}\bar g^{-1},
	\]
	and hence
	$
	\bar g^{-1}\F_{q^n}\subseteq \F_{q^n}\bar g^{-1}.
	$
	By symmetry, equality follows, so that
	$
	\bar g^{-1}\in \Cen_{R_F}(\F_{q^n})^\times.$
	
	We now characterize the elements of \(\Cen_{R_F}(\F_{q^n})^\times\). We first show that \(X+RF(X^n)\) is invertible in \(R_F\). Since \(F(Y)\neq Y\), the constant term of \(F(Y)\) is nonzero. Hence the skew polynomials \(F(X^n)\) and \(X\) are right coprime in \(R\), that is,
	$
	\gcrd(X,F(X^n))=1$.
	By Theorem \ref{th:rankpolynomial_rephrased}, we then have
	\[
	\rk\bigl(X+RF(X^n)\bigr)
	=
	\frac{1}{s}\bigl(\deg(F(X^n))-\deg(\gcrd(X,F(X^n)))\bigr)
	=
	\frac{1}{s}(ns-0)
	=
	n.
	\]
	Thus \(X+RF(X^n)\) has full rank in \(R_F\cong M_n(E_F)\), and therefore it is invertible. It follows that
	$
	X^t+RF(X^n)
	$
	is invertible for every \(t\in\{0,\dots,n-1\}\). 
	
	\begin{lemma}\label{lem:units-Cen}
		Let
		$
		\bar g=z(X^n)X^t+RF(X^n)\in \Cen_{R_F}(\F_{q^n}),
		$
		where \(t\in\{0,\dots,n-1\}\) and \(z(Y)\in \F_{q^n}[Y]\). Then \(\bar g\) is invertible if and only if
		$
		\gcd(z(Y),F(Y))=1$
		in \(\F_{q^n}[Y]\).
	\end{lemma}
	
	\begin{proof}
		The above discussion shows that $X^t+RF(X^n)$ is invertible in $R_F$. Consequently, \(\bar g=z(X^n)X^t+RF(X^n)\) is invertible if and only if
		$
		z(X^n)+RF(X^n)
		$
		is invertible. Now consider the map
		\[
		\varphi:\F_{q^n}[Y]\longrightarrow R_F,
		\qquad
		u(Y)\longmapsto u(X^n)+RF(X^n).
		\]
		Since \(X^n\) commutes with every element of \(\F_{q^n}\), this is a ring homomorphism. Moreover,
		\[
		\ker(\varphi)=\fqn[Y]F(Y),
		\]
		because \(u(X^n)\in RF(X^n)\) if and only if \(F(Y)\mid u(Y)\) in \(\F_{q^n}[Y]\). Hence \(\varphi\) induces an isomorphism
		\[
		\F_{q^n}[Y]/(F(Y))
		\;\cong\;
		\{\,u(X^n)+RF(X^n):u(Y)\in \F_{q^n}[Y]\,\}.
		\]
		Under this identification, the element \(z(X^n)+RF(X^n)\) corresponds to the class of \(z(Y)\) modulo \(F(Y)\). Therefore \(z(X^n)+RF(X^n)\) is invertible if and only if the class of \(z(Y)\) is invertible in \(\F_{q^n}[Y]/\fqn[Y]F(Y)\), and this happens if and only if
		$
		\gcd(z(Y),F(Y))=1
		$
		in \(\F_{q^n}[Y]\). This proves the claim.
	\end{proof}

	We next record the shape of the conjugates of \(\alpha X+RF(X^n)\) by elements of
	\(\Cen_{R_F}(\F_{q^n})^\times\). 
	
	\begin{lemma}\label{lm:shapeconjugatesalphax}
		Let \(\bar g\in \Cen_{R_F}(\F_{q^n})^\times\) and let \(\alpha\in \F_{q^n}^\times\). Then there exists a polynomial \(z(Y)\in \F_{q^n}[Y]\), such that
		\[
		\bar g\,\alpha X\,\bar g^{-1}=z(X^n)X+RF(X^n).
		\]
	\end{lemma}
	
	\begin{proof}
		By Proposition \ref{prop:descriptioncentraliser}, there exist polynomials \(u(Y),v(Y)\in \F_{q^n}[Y]\), and integers \(t,r\in\{0,\dots,n-1\}\) such that
		\[
		\bar g=u(X^n)X^t+RF(X^n)
		\qquad\text{and}\qquad
		\bar g^{-1}=v(X^n)X^r+RF(X^n).
		\]
		Since \(\bar g\bar g^{-1}=1\), one necessarily has
		\[
		t+r\equiv 0 \pmod n.
		\]
		Therefore,
		\[
		\bar g\,\alpha X\,\bar g^{-1}
		=
		u(X^n)X^t\alpha Xv(X^n)X^r+RF(X^n).
		\]
		Using the multiplication rule in \(R\), we obtain
		\[
		\bar g\,\alpha X\,\bar g^{-1}
		=
		u(X^n)\sigma^t(\alpha)\,v^{\sigma^{t+1}}(X^n)X^{t+r+1}+RF(X^n),
		\]
		where \(v^{\sigma^{t+1}}(Y)\) is obtained from \(v(Y)\) by applying \(\sigma^{t+1}\) to its coefficients. Since \(t+r\equiv 0 \pmod n\), the exponent \(t+r+1\) is congruent to \(1\) modulo \(n\). Hence there exists a polynomial \(z(Y)\in \F_{q^n}[Y]\), such that
		\[
		\bar g\,\alpha X\,\bar g^{-1}=z(X^n)X+RF(X^n).
		\]
		This proves the claim.
	\end{proof}
	
	The next theorem is the key structural result of this section. Under the assumption \(k\le n/2\), it characterizes the elements \(\bar b\in R_F^\times\) for which the intersections \(\mathcal C\cap \bar b\,\mathcal C\) or \(\mathcal C\cap \mathcal C\,\bar b\) have codimension one over \(\F_{q^n}\). In particular, if such an element is conjugate to a scalar multiple of \(X\) by an element of \(\Cen_{R_F}(\F_{q^n})^\times\), then it must itself be a scalar multiple of \(X\).

	\begin{theorem}\label{thm:b-linear}
		Assume that \(s>1\) and \(1\le k\le n/2\), and set
		$
		\mathcal C:=\mathcal S_k(F)\subseteq R_F$. Let \(\bar b\in R_F^\times\).
		
		\begin{enumerate}[label=\roman*)]
			\item If
			\[
			|\mathcal C\cap \bar b\,\mathcal C|=q^{n(sk-1)},
			\]
			then there exist \(\alpha,\beta,\alpha',\beta'\in\F_{q^n}\), with
			\((\alpha,\beta)\neq(0,0)\) and \((\alpha',\beta')\neq(0,0)\), such that
			\[
			\bar b=
			\bigl(\alpha+\beta X+RF(X^n)\bigr)
			\bigl(\alpha'+\beta'X+RF(X^n)\bigr)^{-1}.
			\]
			
			\item If
			\[
			|\mathcal C\cap \mathcal C\,\bar b|=q^{n(sk-1)},
			\]
			then there exist \(\alpha,\beta,\alpha',\beta'\in\F_{q^n}\), with
			\((\alpha,\beta)\neq(0,0)\) and \((\alpha',\beta')\neq(0,0)\), such that
			\[
			\bar b=
			\bigl(\alpha'+\beta'X+RF(X^n)\bigr)^{-1}
			\bigl(\alpha+\beta X+RF(X^n)\bigr).
			\]
		\end{enumerate}
	\end{theorem}
	
	\begin{proof}
		We prove \emph{i)}; the proof of \emph{ii)} is entirely analogous, exchanging left and right multiplication throughout.
		
		Set
		\[
		\mathcal D:=\bar b^{-1}\mathcal C\cap \mathcal C.
		\]
		Since left multiplication by \(\bar b\) is a bijection of \(R_F\), we have
		\[
		|\mathcal D|
		=
		|\bar b^{-1}\mathcal C\cap \mathcal C|
		=
		|\mathcal C\cap \bar b\,\mathcal C|
		=
		q^{n(sk-1)}.
		\]
		Moreover, \(\mathcal D\) is a right \(\F_{q^n}\)-subspace of \(R_F\), and therefore
		\[
		\dim_{\F_{q^n}}(\mathcal D)=sk-1.
		\]
		
		Let \(t\) denote the highest degree of a nonzero element of \(\mathcal D\). Since
		\(\mathcal D\subseteq \mathcal C\), one has \(t\le sk-1\). We distinguish two cases.
		
		\smallskip
		\noindent\textbf{Case 1: \(t\le sk-2\).}
		Since \(\dim_{\F_{q^n}}(\mathcal D)=sk-1\), it follows that necessarily \(t=sk-2\) and
		\[
		\mathcal D
		=
		\langle 1+RF(X^n),\,X+RF(X^n),\,\dots,\,X^{sk-2}+RF(X^n)\rangle_{\F_{q^n}}.
		\]
		In particular,
		\[
		1+RF(X^n),\,X+RF(X^n),\,\dots,\,X^{sk-2}+RF(X^n)\in \bar b^{-1}\mathcal C,
		\]
		and therefore
		\[
		\bar b,\ \bar bX,\ \dots,\ \bar bX^{sk-2}\in \mathcal C.
		\]
		Since \(\bar b\in \mathcal C\), we have \(\deg(\bar b)\le sk-1\). Moreover,
		\[
		\deg(\bar bX^{sk-2})\le (sk-1)+(sk-2)=2sk-3<ns,
		\]
		because \(k\le n/2\). Thus no reduction modulo \(F(X^n)\) occurs in the product
		\(\bar bX^{sk-2}\), and from \(\bar bX^{sk-2}\in \mathcal C\) we obtain
		\[
		\deg(\bar b)+(sk-2)\le sk-1.
		\]
		Hence \(\deg(\bar b)\le 1\).
		
		\smallskip
		\noindent\textbf{Case 2: \(t=sk-1\).}
		Consider the \(2\)-dimensional \(\F_{q^n}\)-subspace
		\[
		U:=\langle 1+RF(X^n),\,X+RF(X^n)\rangle_{\F_{q^n}}\subseteq \mathcal C.
		\]
		By Grassmann's formula,
		\[
		\dim_{\F_{q^n}}(\mathcal D\cap U)
		\ge
		\dim_{\F_{q^n}}(\mathcal D)+\dim_{\F_{q^n}}(U)-\dim_{\F_{q^n}}(\mathcal C)
		=
		(sk-1)+2-sk=1.
		\]
		Hence \(\mathcal D\cap U\neq\{0\}\). We claim that \(1+RF(X^n)\notin \mathcal D\). Indeed, if \(1+RF(X^n)\in \mathcal D\), then
		\(\bar b\in \mathcal C\), so \(\deg(\bar b)\le sk-1\). Choose \(\bar d\in \mathcal D\) with
		\(\deg(\bar d)=sk-1\). Since \(\bar b\bar d\in \mathcal C\) and
		\[
		\deg(\bar b\bar d)\le (sk-1)+(sk-1)=2sk-2<ns,
		\]
		no reduction modulo \(F(X^n)\) occurs. Therefore
		\[
		\deg(\bar b)+\deg(\bar d)\le sk-1,
		\]
		which forces \(\deg(\bar b)=0\). Thus \(\bar b\in \F_{q^n}^\times\), and hence
		\(\bar b^{-1}\mathcal C=\mathcal C\), so \(\mathcal D=\mathcal C\), contrary to
		\(\dim_{\F_{q^n}}(\mathcal D)=sk-1\). This proves the claim.
		
		It follows that there exists
		\[
		\bar u=\gamma+\delta X+RF(X^n)\in \mathcal D\cap U
		\]
		with \(\delta\neq 0\). We next show that \(\deg(\bar b)=ns-1\). Assume, to the contrary, that
		\(\deg(\bar b)\le ns-2\). Since \(\bar u\in \mathcal D\), we have \(\bar b\bar u\in \mathcal C\).
		Because \(\deg(\bar u)=1\) and \(\deg(\bar b)\le ns-2\), no reduction modulo \(F(X^n)\) occurs in
		\(\bar b\bar u\), and therefore
		\[
		\deg(\bar b)+1\le sk-1.
		\]
		Hence \(\deg(\bar b)\le sk-2\). Now choose again \(\bar d\in \mathcal D\) with
		\(\deg(\bar d)=sk-1\). Then \(\bar b\bar d\in \mathcal C\), and
		\[
		\deg(\bar b\bar d)\le (sk-2)+(sk-1)=2sk-3<ns,
		\]
		so again no reduction occurs. Thus
		\[
		\deg(\bar b)+\deg(\bar d)\le sk-1,
		\]
		whence \(\deg(\bar b)=0\). As above, this implies \(\bar b\in \F_{q^n}^\times\), hence
		\(\mathcal D=\mathcal C\), a contradiction. Therefore
		\[
		\deg(\bar b)=ns-1.
		\]
		
		Set
		\[
		W:=\langle 1+RF(X^n),\,\bar b\rangle_{\F_{q^n}},
		\qquad
		U_i:=\sum_{j=0}^{i}WX^j
		\qquad (0\le i\le sk-1),
		\]
		where all spans are taken on the right over \(\F_{q^n}\). Since
		\[
		\mathcal C
		=
		\langle 1+RF(X^n),\,X+RF(X^n),\,\ldots,\,X^{sk-1}+RF(X^n)\rangle_{\F_{q^n}},
		\]
		it follows that
		\[
		U_{sk-1}=\mathcal C+\bar b\,\mathcal C.
		\]
		Moreover,
		\[
		\dim_{\F_{q^n}}(\mathcal C+\bar b\,\mathcal C)
		=
		\dim_{\F_{q^n}}(\mathcal C)
		+
		\dim_{\F_{q^n}}(\bar b\,\mathcal C)
		-
		\dim_{\F_{q^n}}(\mathcal C\cap \bar b\,\mathcal C)
		=
		sk+sk-(sk-1)=sk+1.
		\]
		Hence
		\[
		\dim_{\F_{q^n}}(U_{sk-1})=sk+1.
		\]
		
		Since \(\deg(\bar b)=ns-1\), the elements \(1+RF(X^n)\) and \(\bar b\) are linearly independent over \(\F_{q^n}\); therefore
		\[
		\dim_{\F_{q^n}}(W)=2.
		\]
		For convenience, write
		\[
		d_i:=\dim_{\F_{q^n}}(U_i),
		\qquad 0\le i\le sk-1.
		\]
		Thus
		\[
		d_0=2
		\qquad\text{and}\qquad
		d_{sk-1}=sk+1.
		\]
		
		We next show that the sequence \((d_i)\) increases by at least one at each step.
		For every \(0\le i\le sk-2\), one has
		\[
		U_{i+1}=U_i+WX^{i+1}\subseteq U_i+U_iX.
		\]
		On the other hand, since \(WX^jX=WX^{j+1}\subseteq U_{i+1}\) for every \(0\le j\le i\), we also have
		\(U_iX\subseteq U_{i+1}\). Hence
		\[
		U_{i+1}=U_i+U_iX.
		\]
		Now, right multiplication by \(X+RF(X^n)\) is an invertible \(\sigma^{-1}\)-semilinear map, since
		\(X+RF(X^n)\) is invertible in \(R_F\). Consequently,
		\[
		\dim_{\F_{q^n}}(U_iX)=\dim_{\F_{q^n}}(U_i)=d_i.
		\]
		
		We claim that
		\[
		U_i\neq U_iX
		\qquad\text{for every }0\le i\le sk-2.
		\]
		Assume, for contradiction, that \(U_i=U_iX\) for some \(i\). Since \(1+RF(X^n)\in W\subseteq U_i\),
		it follows that
		\[
		X+RF(X^n)=(1+RF(X^n))(X+RF(X^n))\in U_iX=U_i.
		\]
		Iterating, we obtain
		\[
		X^m+RF(X^n)\in U_i
		\qquad\text{for all }m\ge 0.
		\]
		In particular,
		\[
		1+RF(X^n),\ X+RF(X^n),\ \ldots,\ X^{ns-1}+RF(X^n)\in U_i.
		\]
		These \(ns\) elements are linearly independent over \(\F_{q^n}\). Hence
		\[
		d_i=\dim_{\F_{q^n}}(U_i)\ge ns.
		\]
		This is impossible, since \(U_i\subseteq U_{sk-1}\) and, because \(s>1\) and \(k\le n/2\),
		\[
		d_{sk-1}=sk+1<ns.
		\]
		This proves the claim.
		
		Therefore, for every \(0\le i\le sk-2\), one has
		\[
		\dim_{\F_{q^n}}(U_i\cap U_iX)\le d_i-1.
		\]
		Applying Grassmann's formula to \(U_{i+1}=U_i+U_iX\), we obtain
		\[
		d_{i+1}
		=
		\dim_{\F_{q^n}}(U_i+U_iX)
		=
		\dim_{\F_{q^n}}(U_i)+\dim_{\F_{q^n}}(U_iX)-\dim_{\F_{q^n}}(U_i\cap U_iX)
		\ge d_i+1.
		\]
		Thus
		\[
		d_i\ge i+2
		\qquad\text{for all }0\le i\le sk-1.
		\]
		Taking \(i=sk-1\) gives
		\[
		d_{sk-1}\ge sk+1.
		\]
		Since we already know that \(d_{sk-1}=sk+1\), all the above inequalities are in fact equalities.
		In particular,
		\[
		d_1=3.
		\]
		Equivalently,
		\[
		\dim_{\F_{q^n}}(W+WX)=3.
		\]
		
		Now
		\[
		W+WX
		=
		\langle 1+RF(X^n),\,\bar b\rangle_{\F_{q^n}}
		+
		\langle X+RF(X^n),\,\bar bX\rangle_{\F_{q^n}}
		=
		L+\bar bL,
		\]
		where
		\[
		L:=\langle 1+RF(X^n),\,X+RF(X^n)\rangle_{\F_{q^n}}.
		\]
		Since \(\dim_{\F_{q^n}}(L)=2\) and left multiplication by \(\bar b\) is an \(\F_{q^n}\)-linear automorphism of the right vector space \(R_F\), we also have
		\[
		\dim_{\F_{q^n}}(\bar bL)=2.
		\]
		Hence, by Grassmann's formula,
		\[
		\dim_{\F_{q^n}}(L\cap \bar bL)=2+2-3=1.
		\]
		
		Choose a nonzero element of \(L\cap \bar bL\). Then there exist
		\(\alpha,\beta,\alpha',\beta'\in\F_{q^n}\), with \((\alpha,\beta)\neq(0,0)\) and
		\((\alpha',\beta')\neq(0,0)\), such that
		\[
		\alpha+\beta X+RF(X^n)
		=
		\bar b\,(\alpha'+\beta'X+RF(X^n)).
		\]
		Since \(s>1\), every nonzero element of degree at most \(1\) in \(R_F\) is invertible. Therefore
		\[
		\bar b=
		\bigl(\alpha+\beta X+RF(X^n)\bigr)
		\bigl(\alpha'+\beta'X+RF(X^n)\bigr)^{-1}.
		\]
		This proves \emph{(i)}.
	\end{proof}
	
	\section{Equivalence}\label{sec:equivalence}
	
	In this section, we address the equivalence problem for the family
	$
	\mathcal S_k(F)$. By means of the tools of
	Section \ref{sec:structuralresults}, the problem is reduced to the comparison of the corresponding
	defining irreducible polynomials.
	
	Throughout this section, we assume that
	$
	s>1.$
	We keep the notation of Section~\ref{sec:semilineartransf}. In particular, for irreducible
	polynomials \(F(Y),G(Y)\in \F_q[Y]\) of degree \(s\), we choose \(w_F,w_G\in \LL^\times\) such that
	\[
	u_F:=\N_{\LL/\F}(w_F)
	\qquad\text{and}\qquad
	u_G:=\N_{\LL/\F}(w_G)
	\]
	are roots of \(F(Y)\) and \(G(Y)\), respectively, and we set
	\[
	T_F:=T_{w_F,\sigma},
	\qquad
	T_G:=T_{w_G,\sigma}.
	\]
	Then \(T_F\) and \(T_G\) are irreducible \(\sigma\)-semilinear operators on \(V\), and
	\[
	T_F^n=u_F\,\id_V,
	\qquad
	T_G^n=u_G\,\id_V.
	\]
	We first observe that the field of scalars naturally associated with \(T_F\) and \(T_G\)
	is the same in both cases. Indeed, since $u_F$ and $u_G$ are roots in $\F$
	of the irreducible polynomials $F(Y),G(Y) \in \mathbb{F}_q[Y]$ of degree $s$, both elements
	generate $\F$ over $\mathbb{F}_q$. On the other hand,
	\[
	T_F^n = u_F\,\mathrm{id}_V
	\qquad\text{and}\qquad
	T_G^n = u_G\,\mathrm{id}_V.
	\]
	Hence
	\[
	\E_{T_F} := \mathbb{F}_q[T_F^n]
	= \mathbb{F}_q(u_F)\,\mathrm{id}_V
	= \F\,\mathrm{id}_V
	\]
	and similarly
	\[
	\E_{T_G} := \mathbb{F}_q[T_G^n]
	= \mathbb{F}_q(u_G)\,\mathrm{id}_V
	= \F\,\mathrm{id}_V.
	\]
	Therefore
	\[
	\E_{T_F} = \E_{T_G},
	\]
	and consequently
	\[
	\operatorname{End}_{\E_{T_F}}(V)=\operatorname{End}_{\E_{T_G}}(V).
	\]
	In particular, both semilinear transformations are naturally defined over the same field of
	scalars 
	\[\E:=\F \, \id_V.\]

	Moreover, by Theorem~\ref{th:polydescription}, we have the following description of the codes $\cS_k(F)$ and $\cS_k(G)$ in terms of semilinear transformations
	\[
	\mathcal C_k(F):=\Psi_{w_F}(\Sk(F))
	=
	\left\{
	\sum_{i=0}^{sk-1}\alpha_i\cdot T_F^i:\alpha_i\in \K
	\right\}
	\subseteq \End_{\E}(V),
	\]
	and
	\[
	\mathcal C_k(G):=\Psi_{w_G}(\Sk(G))
	=
	\left\{
	\sum_{i=0}^{sk-1}\alpha_i\cdot T_G^i:\alpha_i\in \K
	\right\}
	\subseteq \End_{\E}(V).
	\]
	
	To completely determine the equivalence between $\cS_k(F)$ and $\cS_k(G)$, we have to convert this problem into the equivalence between $\cC_k(F)$ and $\cC_k(G)$, because $\cS_k(F)\subseteq R_F$ and $\cS_k(G)\subseteq R_G$ reside in two different structures, whereas $\cC_k(F)$ and $\cC_k(G)$ both lie in $\End_{\E}(V)$, a setting that is more convenient to work with.
	
	We next show that the equivalence problem can be reformulated from the skew-polynomial model to the semilinear one.
	
	\begin{proposition}\label{prop:matrix-vs-semilinear}
		The following hold:
		\begin{enumerate}[label=\textup{(\roman*)}]
			\item The codes \(\Sk(F)\) and \(\Sk(G)\) are linearly equivalent if and only if
			\(\mathcal C_k(F)\) and \(\mathcal C_k(G)\) are linearly equivalent.
			
			\item The codes \(\Sk(F)\) and \(\Sk(G)\) are \(\Gamma\)-equivalent over \(\F_p\) if and only if
			\(\mathcal C_k(F)\) and \(\mathcal C_k(G)\) are \(\Gamma\)-equivalent over \(\F_p\).
		\end{enumerate}
	\end{proposition}
	
	\begin{proof}
		Let
		\[
		\mathcal{M}_F \colon R_F \stackrel{\sim}{\longrightarrow} M_n(\F_{q^s}),
		\qquad
		\mathcal{M}_G \colon R_G \stackrel{\sim}{\longrightarrow} M_n(\F_{q^s})
		\]
		be \(\F_{q^s}\)-algebra isomorphisms. Fix an \(\E\)-basis of \(V\), and let
		\[
		\mathcal{M}\colon \End_{\E}(V)\stackrel{\sim}{\longrightarrow} M_n(\F_{q^s})
		\]
		be the corresponding matrix representation. For each \(H\in\{F,G\}\), consider the commutative diagram
		\[
		\begin{tikzcd}
			R_H \arrow[r,"\Psi_{w_H}"] \arrow[d,"\mathcal{M}_H"'] &
			\End_{\E}(V) \arrow[d,"\mathcal{M}"] \\
			M_n(\F_{q^s}) \arrow[r,"\Theta_H"'] & M_n(\F_{q^s}),
		\end{tikzcd}
		\]
		where
		\[
		\Theta_H:=\mathcal{M}\circ \Psi_{w_H}\circ \mathcal{M}_H^{-1}.
		\]
		Since \(\Psi_{w_H}\), \(\mathcal M_H\), and \(\mathcal M\) are \(\F_{q^s}\)-algebra isomorphisms,
		\(\Theta_H\) is an \(\F_{q^s}\)-algebra automorphism of \(M_n(\F_{q^s})\). Hence, by the
		Skolem--Noether theorem, there exists \(P_H\in \GL_n(\F_{q^s})\) such that
		\[
		\Theta_H(A)=P_HAP_H^{-1}
		\qquad\text{for all }A\in M_n(\F_{q^s}).
		\]
		Therefore,
		\begin{equation} \label{eq:equivalencedef}
			\mathcal M(\C_k(H))
			=
			P_H\,\mathcal M_H(\Sk(H))\,P_H^{-1}.
		\end{equation}
		
		Suppose first that \(\Sk(F)\) and \(\Sk(G)\) are linearly equivalent. Then there exist
		\(U,V\in \GL_n(\F_{q^s})\) such that
		\[
		\mathcal M_G(\Sk(G))=U\,\mathcal M_F(\Sk(F))\,V.
		\]
		It follows that
		\[
		\mathcal M(\mathcal C_k(G))
		=
		(P_GUP_F^{-1})\,\mathcal M(\mathcal C_k(F))\,(P_FVP_G^{-1}),
		\]
		so that \(\mathcal C_k(F)\) and \(\mathcal C_k(G)\) are linearly equivalent. The converse is obtained in the same way, and this proves (i).
		
		For (ii), suppose that \(\Sk(F)\) and \(\Sk(G)\) are \(\Gamma\)-equivalent over \(\F_p\). Then there exist
		\(U,V\in \GL_n(\F_{q^s})\) and \(\rho\in \Aut_{\F_p}(\F_{q^s})\) such that
		\[
		\mathcal M_G(\Sk(G))
		=
		U\,\mathcal M_F(\Sk(F))^\rho\,V.
		\]
		Using \eqref{eq:equivalencedef}, with $H=F$, we obtain
		\[
		\mathcal M_F(\Sk(F))^\rho
		=
		(P_F^{-1})^\rho\,\mathcal M(\mathcal C_k(F))^\rho\,P_F^\rho.
		\]
		Substituting this into the previous identity and using again \eqref{eq:equivalencedef} with $H=G$, yields
		\[
		\mathcal M(\mathcal C_k(G))
		=
		\bigl(P_GU(P_F^{-1})^\rho\bigr)\,
		\mathcal M(\mathcal C_k(F))^\rho\,
		\bigl(P_F^\rho V P_G^{-1}\bigr).
		\]
		Therefore \(\mathcal C_k(F)\) and \(\mathcal C_k(G)\) are \(\Gamma\)-equivalent over \(\F_p\). The converse is similar.
	\end{proof}
	
	\subsection{Linear equivalence}
	Our goal in this part is to prove the following result on the linear equivalence between $\Sk(F)$ and $\Sk(G)$. 	For a monic polynomial \(P(Y)\in\F_q[Y]\) of degree \(s\) with \(P(0)\neq0\), we denote by
	\[
	\widehat P(Y):=P(0)^{-1}Y^sP(Y^{-1})
	\]
	its monic reciprocal polynomial. Note that, if \(u\) is a root of \(P\), then \(u^{-1}\) is a root of
	\(\widehat P\).
	\begin{theorem}\label{thm:comparison-skew-model}
		Let \(1\le k\le n-1\). The codes
		\[
		\Sk(F)\subseteq R_F
		\qquad\text{and}\qquad
		\Sk(G)\subseteq R_G
		\]
		are linearly equivalent if and only if there exists \(\lambda\in \F_q^\times\) such that
		\[
		F(Y)=\lambda^sG(\lambda^{-1}Y),
		\]
		or, only in the case \(n=2\), such that
		\[
		F(Y)=\lambda^s\widehat G(\lambda^{-1}Y).
		\]
	\end{theorem}
	
	By Proposition~\ref{prop:matrix-vs-semilinear}, in order to classify the linear equivalence classes
	of \(\Sk(F)\), it is enough to study the linear equivalence of the codes \(\mathcal C_k(F)\) and
	\(\mathcal C_k(G)\).
	
	We first determine the left and right idealisers of \(\mathcal C_k(F)\) and \(\mathcal C_k(G)\).
	
	\begin{lemma}\label{lem:semilinear-idealizers}
		For \(H\in\{F,G\}\), one has
		\[
		\lid(\C_k(H))=\rid(\C_k(H))=\{\alpha\,\id_V:\alpha\in \K\}.
		\]
	\end{lemma}
	
	\begin{proof}
		Since \(\C_k(H)=\Psi_{w_H}(\Sk(H))\) and
		\(\Psi_{w_H}\colon R_H\to \End_{\E}(V)\) is a ring isomorphism, it preserves left and right
		idealisers. Hence
		\[
		\lid(\C_k(H))=\Psi_{w_H}\bigl(\lid(\Sk(H))\bigr),
		\qquad
		\rid(\C_k(H))=\Psi_{w_H}\bigl(\rid(\Sk(H))\bigr).
		\]
		By Proposition~\ref{prop:invariantsSNsk},
		\[
		\lid(\Sk(H))=\rid(\Sk(H))
		=
		\{\alpha+RH(X^n):\alpha\in \K\}.
		\]
		Applying \(\Psi_{w_H}\), and using
		\[
		\Psi_{w_H}(\alpha+RH(X^n))=\alpha\,\id_V,
		\]
		we obtain the claim.
	\end{proof}
	
	\begin{lemma}\label{lem:equiv-centralizer}
		Assume that there exist \(M,N\in \GL_{\E}(V)\) such that
		\[
		M\mathcal C_k(F) N=\mathcal C_k(G).
		\]
		Then
		\[
		M,N\in \Cen_{\End_{\E}(V)}(\K).
		\]
	\end{lemma}
	
	\begin{proof}
		From
		\[
		M\mathcal C_k(F) N=\mathcal C_k(G),
		\]
		it follows that
		\[
		M\lid(\mathcal C_k(F))=\lid(\mathcal C_k(G))M,
		\qquad
		N\rid(\mathcal C_k(F))=\rid(\mathcal C_k(G))N.
		\]
		By Lemma~\ref{lem:semilinear-idealizers},
		\[
		\lid(\mathcal C_k(F))=\lid(\mathcal C_k(G))=\rid(\mathcal C_k(F))=\rid(\mathcal C_k(G))
		=\{\alpha\,\id_V:\alpha\in \K\}.
		\]
		Hence, we get
		\[
		M,N\in \Cen_{\End_{\E}(V)}(\K). \qedhere
		\]
	\end{proof}
	
	The following proposition is the semilinear refinement of Theorem~\ref{thm:b-linear} that drives the
	proof of the comparison theorem.
	
	\begin{proposition}\label{prop:semilinear-collapse}
		Assume that \(1\le k\le n/2\). Let \(H\in\{F,G\}\). Let \(S\in \End_{\E}(V)\) be an invertible \(\sigma\)-semilinear map with
		respect to the \(\K\)-vector space structure on \(V\). Assume that
		\[
		|\C_k(H)\cap S\C_k(H)|=q^{n(sk-1)}.
		\]
		Then
		\[
		S\in
		\begin{cases}
			\K^\times T_H, & \text{if } n>2,\\[2mm]
			\K^\times T_H\ \cup\ \K^\times T_H^{-1}, & \text{if } n=2.
		\end{cases}
		\]
		Analogously, if
		\[
		|\C_k(H)\cap \C_k(H)S|=q^{n(sk-1)},
		\]
		then again
		\[
		S\in
		\begin{cases}
			\K^\times T_H, & \text{if } n>2,\\[2mm]
			\K^\times T_H\ \cup\ \K^\times T_H^{-1}, & \text{if } n=2.
		\end{cases}
		\]
	\end{proposition}

	\begin{proof}
		We prove the first assertion, the second one being analogous.
		
		Set
		\[
		\overline{b}:=\Psi_{w_H}^{-1}(S)\in R_H^\times.
		\]
		Since \(\Psi_{w_H}\) is an algebra isomorphism,
		\[
		|\C_k(H)\cap S\C_k(H)|
		=
		|\Sk(H)\cap \overline{b}\,\Sk(H)|.
		\]
		By Theorem~\ref{thm:b-linear}, there exist \(\alpha,\beta,\alpha',\beta'\in \K\), with
		\((\alpha,\beta)\neq(0,0)\) and \((\alpha',\beta')\neq(0,0)\), such that
		\[
		\overline{b}
		=
		\bigl(\alpha+\beta X+RH(X^n)\bigr)
		\bigl(\alpha'+\beta'X+RH(X^n)\bigr)^{-1}.
		\]
		Applying \(\Psi_{w_H}\), we obtain
		\[
		S=(\alpha\,\id_V+\beta T_H)(\alpha'\,\id_V+\beta'T_H)^{-1},
		\]
		hence
		\[
		\alpha\,\id_V+\beta T_H
		=
		S(\alpha'\,\id_V+\beta'T_H).
		\]
		Since \(S\) is \(\sigma\)-semilinear, this becomes
		\begin{equation}\label{eq:id+T_H=S+ST_H}
			\alpha\,\id_V+\beta T_H
			=
			\sigma(\alpha')S+\sigma(\beta')ST_H.	
		\end{equation}
		Now \(\alpha\,\id_V\) is \(\K\)-linear, \(\beta T_H\) and \(\sigma(\alpha')S\) are
		\(\sigma\)-semilinear, while \(\sigma(\beta')ST_H\) is \(\sigma^2\)-semilinear. 
		
		When \(n>2\), the
		semilinear types \(1,\sigma,\sigma^2\) are pairwise distinct. By
		Proposition~\ref{prop:semilinear-centralizer}, the decomposition of \(\End_{\E}(V)\) according to
		the semilinear type is direct. Therefore, we obtain
		\[
		\alpha=0,\qquad \beta'=0,\qquad \beta T_H=\sigma(\alpha')S.
		\]
		As \(S\) is invertible, we must have \(\beta\neq 0\) and \(\alpha'\neq 0\). Hence
		\[
		S=\delta T_H
		\qquad\text{with}\qquad
		\delta:=\beta\,\sigma(\alpha')^{-1}\in \K^\times.
		\]
		When \(n=2\), one has \(\sigma^2=\id_{\K}\). Hence the semilinear types occurring in
		\eqref{eq:id+T_H=S+ST_H} are only \(1\) and \(\sigma\). Therefore, we obtain
		\[
		\alpha\,\id_V=\sigma(\beta')ST_H,
		\qquad
		\beta T_H=\sigma(\alpha')S.
		\]
		If \(\alpha'\neq 0\), then the second equality forces \(\beta\neq 0\), since \(S\) is invertible, and
		therefore
		\[
		S=\beta\,\sigma(\alpha')^{-1}T_H\in \K^\times T_H.
		\]
		Assume now that \(\alpha'=0\). Then \(\beta'\neq 0\), and the second equality gives
		\(\beta=0\). Since \(S\) and \(T_H\) are invertible, the first equality forces \(\alpha\neq 0\), and
		\[
		S=\alpha\,\sigma(\beta')^{-1}T_H^{-1}\in \K^\times T_H^{-1}.
		\]
		Thus, in the case \(n=2\), one obtains
		\[
		S\in \K^\times T_H\ \cup\ \K^\times T_H^{-1}.
		\]
		This proves the first assertion in all cases, and the second assertion is obtained in the same way.
	\end{proof}

	\begin{theorem}\label{thm:comparison-F-G}
		Assume that \(1\le k\le n/2\). Then the following are equivalent:
		\begin{enumerate}
			\item[\rm (i)] The codes \(\mathcal C_k(F)\) and \(\mathcal C_k(G)\) are linearly equivalent.
			
			\item[\rm (ii)] There exists \(\lambda\in\F_q^\times\) such that
			\[
			F(Y)=\lambda^sG(\lambda^{-1}Y),
			\]
			or, only in the case \(n=2\), such that
			\[
			F(Y)=\lambda^s\widehat{G}(\lambda^{-1}Y).
			\]
			
			\item[\rm (iii)] There exist roots \(u_F\) of \(F(Y)\) and \(u_G\) of \(G(Y)\) in \(\F\), and an element
			\(\lambda\in\F_q^\times\), such that
			\[
			u_F=\lambda u_G,
			\]
			or, only in the case \(n=2\), such that
			\[
			u_F=\lambda u_G^{-1}.
			\]
		\end{enumerate}
	\end{theorem}

	\begin{proof}
		We prove \((i)\Rightarrow(iii)\), \((iii)\Leftrightarrow(ii)\), and \((iii)\Rightarrow(i)\).
		
		\medskip
		
		\noindent
		\((i)\Rightarrow(iii)\).
		Assume that
		\[
		M\mathcal C_k(F) N=\mathcal C_k(G)
		\]
		for some \(M,N\in \GL_{\E}(V)\). By Lemma~\ref{lem:equiv-centralizer},
		\[
		M,N\in \Cen_{\End_\E(V)}(\K).
		\]
		Consider the element \(X+RF(X^n)\in R_F\). We have
		\[
		\left|\Sk(F)\cap (X+RF(X^n))\,\Sk(F)\right|=q^{n(sk-1)}.
		\]
		Applying \(\Psi_{w_F}\), we obtain
		\[
		|\mathcal C_k(F)\cap T_F\mathcal C_k(F)|=q^{n(sk-1)}.
		\]
		Therefore
		\[
		|M(\mathcal C_k(F)\cap T_F\mathcal C_k(F))N|=q^{n(sk-1)}.
		\]
		Since \(M\mathcal C_k(F)N=\mathcal C_k(G)\), this yields
		\[
		|\mathcal C_k(G)\cap (MT_FM^{-1})\mathcal C_k(G)|=q^{n(sk-1)}.
		\]
		Now \(MT_FM^{-1}\) is again an invertible \(\sigma\)-semilinear operator on \(V\), so by
		Proposition~\ref{prop:semilinear-collapse}, if $n>2$, there exists \(\delta\in \K^\times\) such that
		\[
		MT_FM^{-1}=\delta T_G.
		\]
		Raising both sides to the \(n\)-th power, we obtain
		\[
		u_F\,\id_V
		=
		MT_F^nM^{-1}
		=
		(\delta T_G)^n
		=
		\N_{\K/\F_q}(\delta)\,u_G\,\id_V.
		\]
		Hence
		\[
		u_F=\lambda u_G
		\qquad\text{with}\qquad
		\lambda:=\N_{\K/\F_q}(\delta)\in \F_q^\times.
		\]
		
		Assume now that \(n=2\). Then Proposition~\ref{prop:semilinear-collapse} gives
		\[
		MT_FM^{-1}\in \K^\times T_G\ \cup\ \K^\times T_G^{-1}.
		\]
		If \(MT_FM^{-1}=\delta T_G\), then, as above,
		\[
		u_F=\N_{\K/\F_q}(\delta)u_G.
		\]
		If instead \(MT_FM^{-1}=\delta T_G^{-1}\), then
		\[
		u_F\,\id_V
		=
		MT_F^2M^{-1}
		=
		(\delta T_G^{-1})^2
		=
		\N_{\K/\F_q}(\delta)T_G^{-2}
		=
		\N_{\K/\F_q}(\delta)u_G^{-1}\id_V.
		\]
		Thus in the case \(n=2\) one obtains
		\[
		u_F=\lambda u_G
		\quad\text{or}\quad
		u_F=\lambda u_G^{-1}
		\]
		for some \(\lambda\in\F_q^\times\).
		
		\medskip
		
		\noindent
		
		\((iii)\Rightarrow(ii)\).
		Assume that
		\[
		u_F=\lambda u_G
		\qquad\text{for some }\lambda\in \F_q^\times,
		\]
		where \(u_F\) is a root of \(F(Y)\) and \(u_G\) is a root of \(G(Y)\) in \(\F\). Then
		\[
		0=G(u_G)=G(\lambda^{-1}u_F),
		\]
		so the polynomial
		\[
		H(Y):=\lambda^sG(\lambda^{-1}Y)
		\]
		is monic, belongs to \(\F_q[Y]\), has degree \(s\), and vanishes at \(u_F\). Since \(F(Y)\) is the
		minimal polynomial of \(u_F\) over \(\F_q\) and has the same degree, we conclude that
		\[
		F(Y)=H(Y)=\lambda^sG(\lambda^{-1}Y).
		\]
		Assume now that the additional case occurring for \(n=2\) holds, namely
		\[
		u_F=\lambda u_G^{-1}
		\qquad\text{for some }\lambda\in\F_q^\times.
		\]
		Since \(u_G^{-1}\) is a root of the reciprocal polynomial \(\widehat G\), we have
		$
		\widehat G(\lambda^{-1}u_F)=0$. Therefore the polynomial
		\[
		H(Y):=\lambda^s\widehat G(\lambda^{-1}Y)
		\]
		is monic, belongs to \(\F_q[Y]\), has degree \(s\), and vanishes at \(u_F\). Again, by the
		minimality of \(F(Y)\), we obtain
		\[
		F(Y)=H(Y)=\lambda^s\widehat G(\lambda^{-1}Y).
		\]
		\medskip
		
		\noindent
		\((ii)\Rightarrow(iii)\).
		Conversely, if
		\[
		F(Y)=\lambda^sG(\lambda^{-1}Y)
		\qquad\text{for some }\lambda\in \F_q^\times,
		\]
		and \(u_G\) is any root of \(G(Y)\), then
		\[
		F(\lambda u_G)=\lambda^sG(u_G)=0,
		\]
		hence \(u_F:=\lambda u_G\) is a root of \(F(Y)\). In the remaining case, which may occur only when \(n=2\), suppose that
		\[
		F(Y)=\lambda^s\widehat G(\lambda^{-1}Y)
		\qquad\text{for some }\lambda\in\F_q^\times.
		\]
		Let \(u_G\) be any root of \(G(Y)\). Then \(u_G^{-1}\) is a root of \(\widehat G(Y)\), and hence
		\[
		F(\lambda u_G^{-1})
		=
		\lambda^s\widehat G(u_G^{-1})
		=
		0.
		\]
		Thus \(u_F:=\lambda u_G^{-1}\) is a root of \(F(Y)\).
		
		\medskip
		
		\noindent
		\((iii)\Rightarrow(i)\).
		Assume that
		\[
		u_F=\lambda u_G
		\]
		for some \(\lambda\in \F_q^\times\). Since the norm map
		\[
		\N_{\K/\F_q}:\K^\times\longrightarrow \F_q^\times
		\]
		is surjective, there exists \(\delta\in \K^\times\) such that
		\[
		\N_{\K/\F_q}(\delta)=\lambda.
		\]
		Then
		\[
		(\delta T_G)^n
		=
		\N_{\K/\F_q}(\delta)\,T_G^n
		=
		\lambda u_G\,\id_V
		=
		u_F\,\id_V
		=
		T_F^n.
		\]
		Now \(T_F\) and \(\delta T_G\) are both irreducible \(\sigma\)-semilinear operators on the
		\(\K\)-space \(V\). By Theorem~\ref{th:classificationsirreducible}~(b), there exists
		\(M\in \GL_\K(V)\) such that
		\[
		MT_FM^{-1}=\delta T_G.
		\]
		We claim that in fact \(M\in \GL_\E(V)\). Indeed,
		\[
		MT_F^nM^{-1}
		=
		(\delta T_G)^n
		=
		T_F^n.
		\]
		Hence \(M\) commutes with \(T_F^n=u_F\,\id_V\). Since \(u_F\) generates \(\F=\F_{q^s}\) over
		\(\F_q\), we have
		\[
		\E=\E_{T_F}=\F_q[T_F^n]=\F\,\id_V.
		\]
		Therefore \(M\) commutes with every scalar multiplication by an element of \(\F\), and so
		\[
		M\in \GL_\E(V).
		\]
		For every \(i\ge 0\), we then have
		\[
		MT_F^iM^{-1}=(\delta T_G)^i=\delta_iT_G^i
		\]
		for some \(\delta_i\in \K^\times\). Consequently,
		\[
		M\mathcal C_k(F)M^{-1}
		=
		\left\{
		\sum_{i=0}^{sk-1}\alpha_i\delta_i\cdot T_G^i:\alpha_i\in \K
		\right\}.
		\]
		Since each \(\delta_i\) is nonzero, the set of all coefficients \(\alpha_i\delta_i\) still runs
		through all of \(\K\). Thus
		\[
		M\mathcal C_k(F)M^{-1}
		=
		\left\{
		\sum_{i=0}^{sk-1}\beta_i\cdot T_G^i:\beta_i\in \K
		\right\}
		=
		\mathcal C_k(G).
		\]
		Hence \(\mathcal C_k(F)\) and \(\mathcal C_k(G)\) are linearly equivalent. It remains to consider the additional case which may occur when \(n=2\), namely
		$
		u_F=\lambda u_G^{-1},$ for some $\lambda\in\F_q^\times.
		$
		As \(n=2\), the operator \(T_G^{-1}\) is again \(\sigma\)-semilinear. Moreover, as before,
		\[
		(\delta T_G^{-1})^2
		=
		\N_{\K/\F_q}(\delta)T_G^{-2}
		=
		\lambda u_G^{-1}\id_V
		=
		u_F\id_V
		=
		T_F^2.
		\]
		The operator \(T_G^{-1}\) is irreducible whenever \(T_G\) is irreducible. Hence
		\(T_F\) and \(\delta T_G^{-1}\) are irreducible \(\sigma\)-semilinear operators on the
		\(\K\)-space \(V\) with the same square. By Theorem~\ref{th:classificationsirreducible}~(b), there exists \(M\in\GL_\K(V)\) such that
		\[
		MT_FM^{-1}=\delta T_G^{-1}.
		\]
		We claim, as above, that \(M\in\GL_\E(V)\). Indeed,
		\[
		MT_F^2M^{-1}
		=
		(\delta T_G^{-1})^2
		=
		T_F^2.
		\]
		Thus \(M\) commutes with \(T_F^2=u_F\id_V\). Since \(u_F\) generates \(\F=\F_{q^s}\) over
		\(\F_q\), it follows that
		$
		M\in\GL_\E(V).$ Since \(n=2\) and \(1\le k\le n/2\), we have \(k=1\). For every \(i\ge0\), there exists
		\(\delta_i\in\K^\times\) such that
		\[
		MT_F^iM^{-1}=(\delta T_G^{-1})^i=\delta_iT_G^{-i}.
		\]
		Therefore
		\[
		\begin{aligned}
			M\mathcal C_1(F)M^{-1}T_G^{s-1}
			&=
			\left\{
			\sum_{i=0}^{s-1}\alpha_i\delta_iT_G^{-i}T_G^{s-1}:\alpha_i\in\K
			\right\}  \\
			&=
			\left\{
			\sum_{i=0}^{s-1}\alpha_i\delta_iT_G^{s-1-i}:\alpha_i\in\K
			\right\}  \\
			&=
			\mathcal C_1(G).
		\end{aligned}
		\]
		Hence \(\mathcal C_1(F)\) and \(\mathcal C_1(G)\) are linearly equivalent.
	\end{proof}
	
	Finally, we are ready to prove Theorem \ref{thm:comparison-skew-model}.	
	\begin{proof}[Proof of Theorem \ref{thm:comparison-skew-model}]
		If \(1\le k\le n/2\), the assertion follows immediately from
		Proposition~\ref{prop:matrix-vs-semilinear} and Theorem~\ref{thm:comparison-F-G}.
		
		Assume now that \(k>n/2\). Then necessarily \(n>2\), and hence the alternative involving the
		reciprocal polynomial cannot occur. First suppose that there exists \(\lambda\in \F_q^\times\) such that
		\[
		F(Y)=\lambda^sG(\lambda^{-1}Y).
		\]
		Then the argument proving \((iii)\Rightarrow(i)\) in Theorem~\ref{thm:comparison-F-G} shows that the
		semilinear codes \(\mathcal C_k(F)\) and \(\mathcal C_k(G)\) are linearly equivalent. Hence, by
		Proposition~\ref{prop:matrix-vs-semilinear}, the codes \(\Sk(F)\) and \(\Sk(G)\) are linearly
		equivalent.
		
		Conversely, assume that \(\Sk(F)\) and \(\Sk(G)\) are linearly equivalent. Let
		\[
		\mathcal M_F:R_F\stackrel{\sim}{\longrightarrow} M_n(\F_{q^s}),
		\qquad
		\mathcal M_G:R_G\stackrel{\sim}{\longrightarrow} M_n(\F_{q^s})
		\]
		be \(\F_{q^s}\)-algebra isomorphisms, and set
		\[
		\mathcal D_F:=\mathcal M_F(\Sk(F)),
		\qquad
		\mathcal D_G:=\mathcal M_G(\Sk(G)).
		\]
		By assumption, there exist \(U,V\in \GL_n(\F_{q^s})\) such that
		\[
		\mathcal D_G=U\mathcal D_FV.
		\]
		Taking Delsarte duals with respect to the trace bilinear form on \(M_n(\F_{q^s})\), we obtain
		\[
		\mathcal D_G^\perp=V^{-\top}\mathcal D_F^\perp U^{-\top}.
		\]
		Hence \(\mathcal D_F^\perp\) and \(\mathcal D_G^\perp\) are linearly equivalent, and therefore so
		are the dual codes \(\Sk(F)^\perp\) and \(\Sk(G)^\perp\).
		
		By \cite[Proposition~5.9]{gomez2025adjoint}, the dual code \(\Sk(F)^\perp\) is
		equivalent to \(\mathcal S_{n,s,n-k}(\widehat F)\), and similarly \(\Sk(G)^\perp\) is
		equivalent to \(\mathcal S_{n,s,n-k}(\widehat G)\), where \(\widehat F\) and \(\widehat G\) denote
		the monic reciprocal polynomials associated with \(F\) and \(G\), respectively. Since
		\(n-k<n/2\), the first part of the proof applies to these codes. Therefore there exists
		\(\mu\in \F_q^\times\) such that
		\[
		\widehat F(Y)=\mu^s\widehat G(\mu^{-1}Y).
		\]
		By the definition of the monic reciprocal polynomial, the correspondence
		\(P(Y)\mapsto \widehat P(Y)\) is involutory and commutes with the transformation
		\(P(Y)\mapsto \mu^sP(\mu^{-1}Y)\). Applying this to the last identity, we obtain
		\[
		F(Y)=\mu^{-s}G(\mu Y).
		\]
		Equivalently, setting \(\lambda:=\mu^{-1}\), this becomes
		\[
		F(Y)=\lambda^sG(\lambda^{-1}Y).
		\]
		This proves the converse implication and completes the proof.
	\end{proof}
	
	\subsection{Semilinear equivalence over \(\F_p\)}
	
	We now pass to semilinear equivalence. The key point is that twisting by a field automorphism of
	\(\E=\F_{q^s}\) replaces the defining polynomial \(F\) by its coefficientwise conjugate over
	\(\F_p\). \newline 
	For \(\tau\in \Gal(\F_q/\F_p)\) and \(H(Y)=\sum_i a_iY^i\in \F_q[Y]\), we write
	\[
	H^\tau(Y):=\sum_i \tau(a_i)Y^i.
	\]
	
	The next lemma shows this coefficientwise Galois action as a semilinear
	conjugacy of the associated irreducible semilinear transformation.
	
	\begin{lemma}\label{lem:prescribed-semilinear-conjugacy}
		Let \(\rho\in\Aut_{\F_p}(\E)\), and set
		$
		\tau:=\rho|_{\F_q}\in\Gal(\F_q/\F_p).$
		Let \(\tau_\K\in\Aut_{\F_p}(\K)\) be such that $
		\tau_{\K}|_{\F_q}=\tau$. Then there exists an invertible \(\tau_\K\)-semilinear
		map
		$
		N\colon V\longrightarrow V
		$
		with respect to the \(\K\)-vector space structure such that
		\[
		N T_F N^{-1}=T_{F^\tau},
		\]
		where \(T_{F^\tau}\) is such that
		$
		T_{F^\tau}^n=\rho(u_F)\,\id_V.$
	\end{lemma}
	
	\begin{proof}
		Choose any invertible \(\tau_\K\)-semilinear map
		$
		\Lambda \colon V\longrightarrow V$
		with respect to the \(\K\)-vector space structure. Therefore
		\[
		T':=\Lambda T_F \Lambda^{-1}
		\]
		is again \(\sigma\)-semilinear over \(\K\).
		We claim that the minimal polynomial of \((T')^n\) over \(\F_q\) is
		\(F^\tau\). Indeed, write
		\[
		F(Y)=\sum_i F_iY^i,\qquad F_i\in\F_q.
		\]
		Since \(F(T_F^n)=0\), we have
		$
		0=\Lambda F(T_F^n)\Lambda^{-1}$.
		Now \(\Lambda\) is \(\tau_\K\)-semilinear over \(\K\), and
		\(\tau_\K|_{\F_q}=\tau\). Hence
		\[
		\Lambda(F_i\,\id_V)\Lambda^{-1}=\tau(F_i)\,\id_V
		\qquad(F_i\in\F_q).
		\]
		Therefore
		\[
		0=\Lambda F(T_F^n)\Lambda^{-1}
		=
		\sum_i \tau(F_i)(\Lambda T_F^n\Lambda^{-1})^i
		=
		F^\tau((T')^n).
		\]
		Since \(F^\tau\) is irreducible and has the same degree as \(F\), it is the
		minimal polynomial of \((T')^n\) over \(\F_q\). On the other hand, by our choice of \(T_{F^\tau}\),
		\[
		T_{F^\tau}^n=\rho(u_F)\,\id_V,
		\]
		and \(\rho(u_F)\) is a root of \(F^\tau\). Hence the minimal polynomial of
		\(T_{F^\tau}^n\) over \(\F_q\) is also \(F^\tau\). Thus \(T'\) and \(T_{F^\tau}\) are irreducible \(\sigma\)-semilinear
		operators whose \(n\)-th powers have the same minimal polynomial over
		\(\F_q\). By 
		\cite[Theorem~2.10]{Dempwolff2010onirreducible}, there exists
		\(M \in\GL_\K(V)\) such that
		\[
		M T' M^{-1}=T_{F^\tau}.
		\]
		Set
		\[
		N:=M \Lambda.
		\]
		Since \(M\) is \(\K\)-linear and \(\Lambda\) is \(\tau_\K\)-semilinear, the map
		\(N\) is \(\tau_\K\)-semilinear over \(\K\). Moreover,
		\[
		N T_F N^{-1}
		=
		M \Lambda T_F \Lambda^{-1}M^{-1}
		=
		M T'M^{-1}
		=
		T_{F^\tau}.
		\]
		This proves the claim.
	\end{proof}
	Let \(\mathcal B=(e_1,\dots,e_n)\) be a fixed \(\E\)-basis of \(V\). For
	\(\rho\in \Aut_{\F_p}(\E)\), let
	\[
	L_\rho\colon V\to V
	\]
	be the unique \(\rho\)-semilinear map such that
	\[
	L_\rho(e_i)=e_i
	\qquad\text{for all }i=1,\dots,n.
	\]
	Equivalently,
	\[
	L_\rho\Bigl(\sum_{i=1}^n x_i e_i\Bigr)=\sum_{i=1}^n \rho(x_i)e_i
	\qquad\text{for all }x_1,\dots,x_n\in \E.
	\]
	For \(A\in \End_{\E}(V)\), we define
	\[
	A^\rho:=L_\rho A L_\rho^{-1},
	\]
	and for a subset \(\mathcal D\subseteq \End_{\E}(V)\), we set
	\[
	\mathcal D^\rho:=\{A^\rho:A\in\mathcal D\}=L_\rho \mathcal D L_\rho^{-1}.
	\]
	In particular, \(\mathcal C_k(F)^\rho\) denotes the image of \(\mathcal C_k(F)\) under this conjugation.
	Equivalently, after identifying \(\End_{\E}(V)\) with \(M_n(\E)\) via the basis \(\mathcal B\),
	the operator \(A^\rho\) is obtained by applying \(\rho\) entrywise to the matrix of \(A\).

	\begin{lemma}\label{lem:twist-to-conjugate-polynomial}
		Let \(\rho\in\Aut_{\F_p}(\E)\), and set
		$
		\tau:=\rho|_{\F_q}\in\Gal(\F_q/\F_p).$
		Then the codes \(\mathcal C_k(F)^\rho\) and
		\(\mathcal C_k(F^\tau)\) are linearly equivalent.
	\end{lemma}
	
	\begin{proof}
		Let \(u_F\in\E\) be the root of \(F(Y)\) used in the construction of
		\(T_F\), so that
		$
		T_F^n=u_F\,\id_V$.
		Since \(F(u_F)=0\), applying \(\rho\) gives
		\[
		0=\rho(F(u_F))=F^\tau(\rho(u_F)).
		\]
		Thus
		\[
		u_{F^\tau}:=\rho(u_F)
		\]
		is a root of \(F^\tau(Y)\). We choose \(T_{F^\tau}\) to be the
		irreducible \(\sigma\)-semilinear operator associated with the root
		\(u_{F^\tau}\), so that
		\[
		T_{F^\tau}^n=u_{F^\tau}\,\id_V=\rho(u_F)\,\id_V.
		\]
		Let \(\tau_{\K}\in\Aut_{\F_p}(\K)\) be an extension of \(\tau\), that is,
		$
		\tau_{\K}|_{\F_q}=\tau$. By Lemma \ref{lem:prescribed-semilinear-conjugacy}, there exists
		an invertible \(\tau_{\K}\)-semilinear map
		$
		N\colon V\rightarrow V$
		with respect to the \(\K\)-vector space structure such that
		\[
		N T_F N^{-1}=T_{F^\tau}.
		\]
		We first record the induced action of \(N\) on the \(\E\)-scalars. Taking
		\(n\)-th powers in
		$
		N T_F N^{-1}=T_{F^\tau}
		$
		gives
		\[
		N T_F^n N^{-1}=T_{F^\tau}^n.
		\]
		Hence
		\[
		N(u_F\,\id_V)N^{-1}=\rho(u_F)\,\id_V,
		\]
		or equivalently
		\[
		N(u_F v)=\rho(u_F)N(v)
		\qquad(v\in V).
		\]
		As a consequence, \(N\) is \(\rho\)-semilinear with respect to the
		\(\E\)-vector space structure. Indeed, since
		$
		\E=\F_q(u_F),$
		every \(z\in\E\) can be written as \(z=h(u_F)\), with
		\(h(Y)\in\F_q[Y]\). Write
		\[
		h(Y)=\sum_j \delta_jY^j,
		\qquad \delta_j\in\F_q.
		\]
		Since \(N\) is \(\tau_{\K}\)-semilinear over \(\K\), and
		\(\tau_{\K}|_{\F_q}=\tau\), we have
		$
		N(\delta_j v)=\tau(\delta_j)N(v)$, with
		$\delta_j\in\F_q$,\ $v\in V$.
		Using also \(N(u_Fv)=\rho(u_F)N(v)\), we obtain
		\[
		\begin{aligned}
			N(zv)
			&=
			N(h(u_F)v)\\
			&=
			N\left(\sum_j \delta_j u_F^jv\right)\\
			&=
			\sum_j \tau(\delta_j)\rho(u_F)^jN(v)\\
			&=
			h^\tau(\rho(u_F))N(v)\\
			&=
			\rho(h(u_F))N(v)\\
			&=
			\rho(z)N(v).
		\end{aligned}
		\]
		Thus
		$
		N(zv)=\rho(z)N(v)$. By definition,
		\[
		\mathcal C_k(F)^\rho
		=
		L_\rho\,\mathcal C_k(F)\,L_\rho^{-1}.
		\]
		The maps \(L_\rho\) and \(N\) are both \(\rho\)-semilinear with respect to
		the \(\E\)-vector space structure. Therefore
		\[
		Q:=N^{-1}L_\rho
		\]
		is \(\E\)-linear. Hence \(Q\in\GL_\E(V)\), and \(L_\rho=NQ\). Consequently
		\[
		\mathcal C_k(F)^\rho
		=
		NQ\,\mathcal C_k(F)\,Q^{-1}N^{-1}.
		\]
		Moreover,
		\[
		\begin{aligned}
			NQ^{-1}N^{-1}\,\mathcal C_k(F)^\rho\,NQN^{-1}
			&=
			NQ^{-1}N^{-1}
			\bigl(
			NQ\,\mathcal C_k(F)\,Q^{-1}N^{-1}
			\bigr)
			NQN^{-1} \\
			&=
			N\,\mathcal C_k(F)\,N^{-1}.
		\end{aligned}
		\] Since \(Q\in\GL_\E(V)\) and \(N\) is \(\rho\)-semilinear over \(\E\), both
		\(NQ^{-1}N^{-1}\) and \(NQN^{-1}\) belong to \(\GL_\E(V)\).
		Thus \(\mathcal C_k(F)^\rho\) is linearly equivalent to
		\(N\mathcal C_k(F)N^{-1}\). Now, since \(N\) is
		\(\tau_\K\)-semilinear over \(\K\), for every \(\alpha\in\K\) we have
		$
		N\alpha N^{-1}=\tau_\K(\alpha).
		$
		Moreover,
		$
		N T_F N^{-1}=T_{F^\tau}.
		$
		Therefore
		\[
		\begin{aligned}
			N\mathcal C_k(F)N^{-1}
			&=
			N
			\left\{
			\sum_{i=0}^{sk-1}\alpha_iT_F^i:
			\alpha_i\in\K
			\right\}
			N^{-1}  \\
			&=
			\left\{
			\sum_{i=0}^{sk-1}
			\tau_\K(\alpha_i)T_{F^\tau}^i:
			\alpha_i\in\K
			\right\} \\
			& =
			\left\{
			\sum_{i=0}^{sk-1}
			\beta_iT_{F^\tau}^i:
			\beta_i\in\K
			\right\} \\
			& =
			\mathcal C_k(F^\tau).
		\end{aligned}
		\]
		We have therefore shown that
		$
		\mathcal C_k(F)^\rho$
		is linearly equivalent to $
		\mathcal C_k(F^\tau),$
		as required.
	\end{proof}

	\begin{theorem}\label{thm:comparison-F-G-semilinear}
		Assume that \(1\le k\le n/2\). Let \(\rho\in \Aut_{\F_p}(\E)\), and set
		$
		\tau:=\rho|_{\F_q}\in \Gal(\F_q/\F_p)$.
		Then the following are equivalent:
		\begin{enumerate}
			\item[\rm (i)] There exist \(M,N\in \GL_{\E}(V)\) such that
			\[
			M\mathcal C_k(F)^\rho N=\mathcal C_k(G);
			\]
			
			\item[\rm (ii)] There exists \(\lambda\in \F_q^\times\) such that
			\[
			F^\tau(Y)=\lambda^sG(\lambda^{-1}Y),
			\]
			or, only in the case \(n=2\), such that
			\[
			F^\tau(Y)=\lambda^s\widehat G(\lambda^{-1}Y).
			\]
			
			\item[\rm (iii)] There exist roots \(u_F\) of \(F(Y)\) and \(u_G\) of \(G(Y)\) in \(\F\), and an
			element \(\lambda\in \F_q^\times\), such that
			\[
			\rho(u_F)=\lambda u_G,
			\]
			or, only in the case \(n=2\), such that
			\[
			\rho(u_F)=\lambda u_G^{-1}.
			\]
		\end{enumerate}
	\end{theorem}
	
	\begin{proof}
		By Lemma~\ref{lem:twist-to-conjugate-polynomial}, the code \(\mathcal C_k(F)^\rho\) is linearly
		equivalent to \(\mathcal C_k(F^\tau)\). Therefore
		$
		M\mathcal C_k(F)^\rho N=\mathcal C_k(G)$
		for some \(M,N\in\GL_\E(V)\) if and only if \(\mathcal C_k(F^\tau)\) and
		\(\mathcal C_k(G)\) are linearly equivalent. Applying Theorem~\ref{thm:comparison-F-G} to the pair \(F^\tau,G\), we obtain that this is
		equivalent to the existence of \(\lambda\in\F_q^\times\) such that
		\[
		F^\tau(Y)=\lambda^sG(\lambda^{-1}Y),
		\]
		or, only in the case \(n=2\), such that
		\[
		F^\tau(Y)=\lambda^s\widehat G(\lambda^{-1}Y).
		\]
		This proves the equivalence between \({\rm (i)}\) and \({\rm (ii)}\).
		
		It remains only to translate this condition in terms of roots. Since \(u_F\) is a root of \(F\),
		the element \(\rho(u_F)\) is a root of \(F^\tau\). Hence, by the same argument used in the proof
		of Theorem~\ref{thm:comparison-F-G}, the identity
		$
		F^\tau(Y)=\lambda^sG(\lambda^{-1}Y)
		$
		is equivalent to the existence of roots \(u_F\) of \(F\) and \(u_G\) of \(G\) such that
		$
		\rho(u_F)=\lambda u_G.
		$
		Likewise, in the additional case \(n=2\), the identity
		$
		F^\tau(Y)=\lambda^s\widehat G(\lambda^{-1}Y)
		$
		is equivalent to the existence of roots \(u_F\) of \(F\) and \(u_G\) of \(G\) such that
		$
		\rho(u_F)=\lambda u_G^{-1},
		$
		because \(u_G^{-1}\) is a root of the reciprocal polynomial \(\widehat G\). This proves
		\({\rm (ii)}\Leftrightarrow{\rm (iii)}\), and completes the proof.
	\end{proof}
	
	\begin{theorem}\label{thm:comparison-skew-model-semilinear}
		Let \(1\le k\le n-1\). The codes
		\[
		\Sk(F)\subseteq R_F
		\qquad\text{and}\qquad
		\Sk(G)\subseteq R_G
		\]
		are \(\Gamma\)-equivalent over \(\F_p\) if and only if there exist
		\[
		\tau\in \Gal(\F_q/\F_p)
		\qquad\text{and}\qquad
		\lambda\in \F_q^\times
		\]
		such that
		\[
		F^\tau(Y)=\lambda^sG(\lambda^{-1}Y),
		\]
		or, only in the case \(n=2\), such that
		\[
		F^\tau(Y)=\lambda^s\widehat G(\lambda^{-1}Y).
		\]
	\end{theorem}
	
	\begin{proof}
		If \(1\le k\le n/2\), the assertion follows immediately from
		Proposition~\ref{prop:matrix-vs-semilinear} and
		Theorem~\ref{thm:comparison-F-G-semilinear}. Assume now that \(k>n/2\). Then necessarily \(n>2\), and hence the alternative involving the
		reciprocal polynomial cannot occur. The same duality argument used in the proof of
		Theorem~\ref{thm:comparison-skew-model} applies verbatim in the semilinear setting, since
		\(\Gamma\)-equivalence over \(\F_p\) is preserved under Delsarte duality. More precisely, if
		\[
		\mathcal D_G=U\mathcal D_F^\rho V
		\]
		for some \(U,V\in \GL_n(\F_{q^s})\) and \(\rho\in \Aut_{\F_p}(\F_{q^s})\), then
		\[
		\mathcal D_G^\perp=V^{-\top}(\mathcal D_F^\perp)^\rho U^{-\top}.
		\]
		Using again \cite[Proposition~5.9]{gomez2025adjoint}, one reduces to the case \(n-k<n/2\), and
		obtains
		\[
		\widehat F^{\,\tau}(Y)=\mu^s\widehat G(\mu^{-1}Y)
		\]
		for some \(\tau\in \Gal(\F_q/\F_p)\) and \(\mu\in \F_q^\times\). The reciprocal involution commutes with coefficientwise conjugation. Moreover,
		for every monic polynomial \(P(Y)\in\F_q[Y]\) of degree \(s\) with
		\(P(0)\neq0\), one has
		\[
		\widehat{\mu^sP(\mu^{-1}Y)}
		=
		\mu^{-s}\widehat P(\mu Y).
		\]
		Applying the reciprocal involution to the previous identity therefore gives
		\[
		F^\tau(Y)=\mu^{-s}G(\mu Y).
		\]
		Equivalently, setting \(\lambda:=\mu^{-1}\), we obtain
		\[
		F^\tau(Y)=\lambda^sG(\lambda^{-1}Y).
		\]
		This proves the claim.
	\end{proof}

	\subsection{Enumerative consequences}

	As immediate consequences of the classification theorems, one obtains the number of
	inequivalent cyclic semifields and MRD codes $S_k(F)$ for fixed $q$, $k$, $s$ and $n$.
	
	Define
	\[
	\mathcal I_{q,s}:=
	\{F(Y)\in\F_q[Y]: F \text{ monic irreducible of degree }s\}.
	\]
	For $(\lambda,\tau) \in \F_q^\times \rtimes \Gal(\F_q/\F_p)$, we define a group action on the elements in $\mathcal{I}_{q,s}$ as follows:
	\begin{equation*}
		%		\label{eq:action.on.Iqs}
		(\lambda,\tau)\cdot F
		:=
		\lambda^sF^\tau(\lambda^{-1}Y),
		\qquad F\in\mathcal{I}_{q,s}.
	\end{equation*}
	Let $\iota$ stand for the reciprocal involution on $\mathcal{I}_{q,s}$, i.e.,
	\[
	\iota(F)=\widehat F, \qquad F\in\mathcal{I}_{q,s}.
	\]
	For $n\geq 2$, define groups $\mathcal{G}$ and $\widehat{\mathcal{G}}$ by
	\[
	\mathcal{G} = 
	\begin{cases}
		\F_q^\times , & n>2\\
		\F_q^\times \rtimes \langle \iota\rangle, & n=2,
	\end{cases}
	\]
	and 
	\[
	\widehat{\mathcal{G}} = \mathcal{G} \rtimes \Gal(\F_q/\F_p).
	\]
	By Theorem~\ref{thm:comparison-skew-model} and 
	Theorem~\ref{thm:comparison-skew-model-semilinear}, one can prove the following results directly.
	\begin{theorem}\label{thm:exact-number-inequivalent}
		Assume that
		\[
		q=p^r,\qquad s>1,\qquad n>1, \qquad 1\le k\le n-1.
		\]
		Let 
		\[
		\varepsilon_n=\begin{cases}
			1, & n>2;\\
			2, & n=2.
		\end{cases}
		\]
		\begin{itemize}
			\item The number of pairwise linearly inequivalent codes \(\Sk(F)\) obtained from $\F_{q^n}[X;\sigma]$ is exactly the number of $\mathcal{G}$-orbits on $\mathcal{I}_{q,s}$. In particular, it is at least
			\[
			\left\lceil \frac{|\mathcal{I}_{q,s}|}{\varepsilon_n (q-1)}\right\rceil .
			\]
			\item The number of pairwise \(\Gamma\)-inequivalent codes \(\Sk(F)\) obtained from $\F_{q^n}[X;\sigma]$ is exactly the number of $\widehat{\mathcal{G}}$-orbits on $\mathcal{I}_{q,s}$. In particular, it is at least
			\[
			\left\lceil \frac{|\mathcal{I}_{q,s}|}{\varepsilon_n r(q-1)}\right\rceil .
			\]
		\end{itemize}
	\end{theorem}
	
	As an immediate consequence, by using
	Proposition~\ref{prop:isotopy-equivalence-right-nucleus}, in the semifield case \(k=1\) we obtain a
	lower bound on the number of isotopy classes of cyclic semifields.
	
	\begin{corollary}\label{cor:lower-bound-nonisotopic-cyclic-semifields}
		Assume that
		\[
		q=p^r,\qquad s>1,\qquad n\ge 2.
		\]
		Then the number of
		pairwise non-isotopic cyclic semifields of order \(q^{ns}\), with left and middle nuclei
		isomorphic to \(\F_{q^n}\) and right nucleus isomorphic to \(\F_{q^s}\), is exactly the number of $\widehat{\mathcal{G}}$-orbits on $\mathcal{I}_{q,s}$. In particular, this number is at least
		\[
		\left\lceil \frac{|\mathcal{I}_{q,s}|}{\varepsilon_n r(q-1)}\right\rceil
		\]
	\end{corollary}

	\begin{remark}
		\begin{enumerate}[label=(\alph*)]
			\item When $n>2$, the number of pairwise non-isotopic cyclic semifields presented in Corollary \ref{cor:lower-bound-nonisotopic-cyclic-semifields} is exactly the upper bound obtained in \cite[Theorem 2]{lavrauw2013semifields}.
			\item %		\label{rem:n-equals-2}
			The case \(n=2\) has been previously studied in \cite{Johnson2009generalization}, where the authors
			consider cyclic semifields which are two-dimensional over their left nucleus, with right and middle
			nuclei isomorphic to \(\F_{q^2}\). In the right-nucleus convention used in the present paper, these
			correspond to the opposite semifields of the algebras considered in
			\cite{Johnson2009generalization}; see also the discussion in \cite[Section~8]{lavrauw2013semifields}.
			
			More precisely, if \(q=p^r\), then the lower bound obtained in
			\cite{Johnson2009generalization} for the number of isotopy classes $A(q,2,s)$ of cyclic semifields in the case \(n=2\) reads as
			\[
			A(q,2,s)\ge
			\frac{q^s-\theta}{2srq(q-1)},
			\]
			where \(\theta\) denotes the number of elements of \(\F_{q^s}\) contained in a proper subfield.
			Since
			\[
			|\mathcal{I}_{q,s}|=\frac{q^s-\theta}{s},
			\]
			the bound provided in Corollary \ref{cor:lower-bound-nonisotopic-cyclic-semifields}, for \(n=2\), gives
			\[
			A(q,2,s)\ge
			\left\lceil
			\frac{|\mathcal{I}_{q,s}|}{2r(q-1)}
			\right\rceil
			=
			\left\lceil
			\frac{q^s-\theta}{2sr(q-1)}
			\right\rceil .
			\]
			Thus, in the rank-two case, our result recovers the known exponential
			growth in \(s\) and improves the previous lower bound by a factor \(q\). The additional factor \(2\) in the denominator is precisely the contribution of the reciprocal involution.
		\end{enumerate}
	\end{remark}

	\section{The full automorphism group and the autotopism group}
	\label{sec:full-autotopism}
	
	In this section, we determine the full automorphism group of the code \(\Sk(F)\). In the
	special case \(k=1\), this also yields the full autotopism group of the associated cyclic
	semifield. Therefore, we write \[ \Aut(\mathcal S_k(F)) = \left\{ (A,B,\rho)\in \GL_n(\E_F)\times \GL_n(\E_F)\times \Aut(\E_F) : A\, \cS_k(F)^{\rho}\,B=\cS_k(F) \right\}, \] and similarly for the group of linear autotopisms of $\Sk(F)$, we have \[ \Aut_{\E}(\cS_k(F)) := \left\{ (A,B)\in \GL_n(\E)\times \GL_n(\E) : A\,\mathcal \cS_k(F)\,B= \cS_k(F) \right\}. \]
	
	By the correspondence between isotopy of semifields and equivalence of the
	associated spread sets described in Section~\ref{subsec:equivalencespreads}, in the case
	\(k=1\) these groups identify with the autotopism groups of the cyclic semifield. More
	precisely, if \(\mathbb S_F\) denotes the cyclic semifield whose spread set is \(\cS_1(F)\), then
	\eqref{eq:reformulationautotopism} and \eqref{eq:reformulationlinearautotopism} yield
	canonical isomorphisms
	\[
	\Aut(\mathbb S_F)\cong \Aut(\cS_1(F))
	\]
	and
	\[
	\Aut_{\E_F}(\mathbb S_F)\cong \Aut_{\E_F}(\cS_1(F)).
	\]
	
	Our first observation is that the two components of an automorphism pair necessarily lie in the unit
	group of the centralizer \(\Cen_{R_F}(\F_{q^n})\). The following result can be obtained by Theorem~\ref{th:polydescription} and Lemma~\ref{lem:equiv-centralizer}.
	
	\begin{lemma}\label{lem:mapsincentraliser}
		Let \(\bar g,\bar h\in R_F^\times\) be such that
		\[
		\bar g\,\Sk(F)\,\bar h=\Sk(F).
		\]
		Then \(\bar g,\bar h\in \Cen_{R_F}(\F_{q^n})^\times\).
	\end{lemma}
	
	\begin{lemma}\label{lem:corrected-conjugate-of-X}
		Assume that \(s>1\) and \(1\le k\le n/2\), and set
		$
		\mathcal C:=\Sk(F)\subseteq R_F$.
		Let \(\bar g\in \Cen_{R_F}(\F_{q^n})^\times\), and set
		\[
		\bar b:=\bar gX\bar g^{-1}.
		\]
		Assume that either
		\[
		|\mathcal C\cap \bar b\,\mathcal C|=q^{n(sk-1)}
		\]
		or
		\[
		|\mathcal C\cap \mathcal C\,\bar b|=q^{n(sk-1)}.
		\]
		Then one of the following alternatives holds:
		\[
		\bar b=\beta X+RF(X^n)
		\]
		for some \(\beta\in\F_{q^n}^\times\), or
		\[
		n=2
		\qquad\text{and}\qquad
		\bar b=\beta X^{-1}+RF(X^2)
		\]
		for some \(\beta\in\F_{q^2}^\times\). Moreover, the second alternative can occur only if \(s=2\) and
		$
		F(Y)=Y^2-\nu$,
		for some \(\nu\in\F_q^\times\).
	\end{lemma}
	
	\begin{proof}
		As $\bar{b}$ defines an invertible $\sigma$-semilinear map under the ring isomorphism defined in Theorem \ref{th:polydescription}, the claim about the shape of $\bar{b}$ follows directly from Proposition \ref{prop:semilinear-collapse}.
		
		It remains to prove the final assertion. Suppose that the second alternative
		occurs. Since
		$
		\bar g\in \Cen_{R_F}(\F_{q^2})^\times,$
		we may write
		\[
		\bar g=w(X^2)+RF(X^2),
		\]
		where \(w(Y)\in\F_{q^2}[Y]\) is invertible modulo \(F(Y)\). Then
		\[
		\bar gX\bar g^{-1}
		=
		w(X^2)\bigl(w^\sigma(X^2)\bigr)^{-1}X.
		\]
		On the other hand, by assumption,
		\[
		\bar gX\bar g^{-1}=\beta X^{-1}+RF(X^2).
		\]
		Therefore, in \(\F_{q^2}[Y]/(F(Y))\), we have
		\[
		w(Y)\bigl(w^\sigma(Y)\bigr)^{-1}
		=
		\beta Y^{-1}.
		\]
		Applying \(\sigma\), we also get
		\[
		w^\sigma(Y)w(Y)^{-1}
		=
		\sigma(\beta)Y^{-1}.
		\]
		Multiplying the two identities gives
		\[
		1=\beta\sigma(\beta)Y^{-2}.
		\]
		Hence
		\[
		Y^2=\beta\sigma(\beta)
		\]
		in \(\F_{q^2}[Y]/(F(Y))\). Since
		$
		\beta\sigma(\beta)\in\F_q^\times,$
		the class of \(Y^2\) modulo \(F(Y)\) belongs to \(\F_q^\times\). Thus \(F(Y)\) divides \(Y^2-\nu\), where
		\[
		\nu:=\beta\sigma(\beta)\in\F_q^\times.
		\]
		Since \(F\) is monic and \(\deg F=s>1\), this forces
		\[
		s=2
		\]
		and
		\[
		F(Y)=Y^2-\nu.
		\]
		The proof is complete.
	\end{proof}
	We next show that an automorphism pair forces the conjugates of \( X+RF(X^n)\) to retain the
	same shape.
	
	\begin{proposition}\label{prop:structuremaps}
		Assume that \(1\le k\le n/2\). Let
		\(\bar g,\bar h\in \Cen_{R_F}(\F_{q^n})^\times\) be such that
		\[
		\bar g\,\Sk(F)\,\bar h=\Sk(F).
		\]
		Then one of the following alternatives holds:
		\[
		\bar gX\bar g^{-1}=\beta X+RF(X^n)
		\]
		for some \(\beta\in\F_{q^n}^\times\), or
		\[
		n=2
		\qquad\text{and}\qquad
		\bar gX\bar g^{-1}=\beta X^{-1}+RF(X^2)
		\]
		for some \(\beta\in\F_{q^2}^\times\).
		Similarly, one of the following alternatives holds:
		\[
		\bar h^{-1}X\bar h=\beta' X+RF(X^n)
		\]
		for some \(\beta'\in\F_{q^n}^\times\), or
		\[
		n=2
		\qquad\text{and}\qquad
		\bar h^{-1}X\bar h=\beta' X^{-1}+RF(X^2)
		\]
		for some \(\beta'\in\F_{q^2}^\times\). Moreover, each of the alternatives involving \(X^{-1}\) can occur only in the
		exceptional case
		\[
		n=2,\qquad s=2,\qquad F(Y)=Y^2-\nu
		\]
		for some \(\nu\in\F_q^\times\).
	\end{proposition}
	
	\begin{proof}
		Set
		\[
		\mathcal S:=\Sk(F)
		\qquad\text{and}\qquad
		\bar u:=X+RF(X^n).
		\]
		Since
		\[
		\mathcal S
		=
		\left\{
		\sum_{i=0}^{sk-1}\alpha_iX^i+RF(X^n):
		\alpha_i\in\F_{q^n}
		\right\},
		\]
		we have
		\[
		\bar u\,\mathcal S
		=
		\left\{
		\sum_{i=0}^{sk-1}\sigma(\alpha_i)X^{i+1}+RF(X^n):
		\alpha_i\in\F_{q^n}
		\right\}.
		\]
		As \(1\le k\le n/2\), we have \(sk\le ns-1\), and hence no reduction modulo
		\(F(X^n)\) occurs in these products. Therefore
		\[
		\mathcal S\cap \bar u\,\mathcal S
		=
		\langle X+RF(X^n),\dots,X^{sk-1}+RF(X^n)\rangle_{\F_{q^n}},
		\]
		and so
		\[
		|\mathcal S\cap \bar u\,\mathcal S|
		=
		q^{n(sk-1)}.
		\]
		
		Now
		\[
		\bar g(\mathcal S\cap \bar u\,\mathcal S)\bar h
		=
		\bar g\mathcal S\bar h
		\cap
		\bar g\bar u\,\mathcal S\bar h.
		\]
		Since \(\bar g\mathcal S\bar h=\mathcal S\), and since
		\[
		\bar g\bar u\,\mathcal S\bar h
		=
		(\bar g\bar u\bar g^{-1})\,\bar g\mathcal S\bar h
		=
		(\bar g\bar u\bar g^{-1})\,\mathcal S,
		\]
		we obtain
		\[
		\bar g(\mathcal S\cap \bar u\,\mathcal S)\bar h
		=
		\mathcal S\cap (\bar g\bar u\bar g^{-1})\,\mathcal S.
		\]
		Since left multiplication by \(\bar g\) and right multiplication by \(\bar h\)
		are bijections, it follows that
		\[
		\left|
		\mathcal S\cap (\bar g\bar u\bar g^{-1})\,\mathcal S
		\right|
		=
		|\mathcal S\cap \bar u\,\mathcal S|
		=
		q^{n(sk-1)}.
		\]
		Applying Lemma~\ref{lem:corrected-conjugate-of-X} to
		\[
		\bar b:=\bar gX\bar g^{-1},
		\]
		we obtain either
		\[
		\bar gX\bar g^{-1}
		=
		\beta X+RF(X^n)
		\]
		for some \(\beta\in\F_{q^n}^\times\), or
		\[
		n=2
		\qquad\text{and}\qquad
		\bar gX\bar g^{-1}
		=
		\beta X^{-1}+RF(X^2)
		\]
		for some \(\beta\in\F_{q^2}^\times\). Moreover, by the same lemma, the second
		alternative can occur only if
		\[
		s=2
		\qquad\text{and}\qquad
		F(Y)=Y^2-\nu
		\]
		for some \(\nu\in\F_q^\times\).
		
		The corresponding assertion for $\bar h$ follows analogously, by using the right-hand analogue of the previous argument.
	\end{proof}
	
	We now classify the invertible elements of \(\Cen_{R_F}(\F_{q^n})\) satisfying the relation
	obtained in Proposition~\ref{prop:structuremaps}.
	
	\begin{lemma}\label{lem:G-description}
		Let \(\bar g\in \Cen_{R_F}(\F_{q^n})^\times\). Suppose that there exists \(\lambda\in\F_{q^n}^\times\) such that
		\[
		\bar g X=\lambda X\,\bar g.
		\]
		Then there exist \(\bar c\in \E_F^\times\), \(\gamma\in\F_{q^n}^\times\), and
		\(t\in\{0,\dots,n-1\}\) such that
		\[
		\bar g=\bar c\,\gamma X^t.
		\]
	\end{lemma}
	
	\begin{proof}
		By Lemma~\ref{lem:units-Cen}, we may write
		\[
		\bar g=z(X^n)X^t+RF(X^n),
		\]
		where \(t\in\{0,\dots,n-1\}\), \(z(Y)\in\F_{q^n}[Y]\) and
		$
		\gcd(z(Y),F(Y))=1$.
		
		Using the multiplication rule in \(R\), we obtain
		\[
		z(X^n)X^{t+1}+RF(X^n)
		=
		\lambda z^\sigma(X^n)X^{t+1}+RF(X^n),
		\]
		hence
		\[
		z(X^n)+RF(X^n)=\lambda z^\sigma(X^n)+RF(X^n).
		\]
		Equivalently,
		\[
		z(Y)=\lambda z^\sigma(Y) \]
		in $\F_{q^n}[Y]/\fqn[Y]F(Y)$. Applying \(\sigma\) repeatedly and multiplying the resulting identities, we obtain
		\[
		z(Y)=\N_{\F_{q^n}/\F_q}(\lambda)\,z(Y),
		\]
		and therefore
		\[
		\N_{\F_{q^n}/\F_q}(\lambda)=1.
		\]
		By Hilbert's Theorem~90, there exists \(\gamma\in\F_{q^n}^\times\) such that
		\[
		\lambda=\gamma\,\sigma(\gamma)^{-1}.
		\]
		It follows that
		\[
		\gamma^{-1}z(Y)=\sigma(\gamma^{-1}z(Y)),
		\]
		so the class of \(\gamma^{-1}z(Y)\) modulo \(F(Y)\) is fixed by \(\sigma\). Since \(F(Y)\in\F_q[Y]\),
		the fixed ring of \(\F_{q^n}[Y]/(F(Y))\) under \(\sigma\) is canonically identified with
		$
		\F_q[Y]/\fq[Y]F(Y)\cong \E_F$.
		Hence there exists \(\bar c\in\E_F^\times\) such that
		\[
		z(X^n)+RF(X^n)=\bar c\,\gamma.
		\]
		Consequently,
		$
		\bar g=\bar c\,\gamma X^t,$
		as claimed.
	\end{proof}

	For later use, set
	\[
	\mathcal G_F
	:=
	\left\{
	\bar c\,\gamma X^t:
	\bar c\in\E_F^\times,\ \gamma\in\F_{q^n}^\times,\ t\in\{0,\dots,n-1\}
	\right\},
	\]
	and
	\[
	\mathcal A_F
	:=
	\left\{
	(\bar g,\bar g^{-1}\lambda):
	\bar g\in\mathcal G_F,\ \lambda\in\F_{q^n}^\times
	\right\}.
	\]
	%	In the exceptional case
	%	\[
	%	n=s=2,\qquad F(Y)=Y^2-\nu,\qquad \nu\in\F_q^\times,
	%	\]
	%	we also set
	%	\[
	%	\mathcal G_F^-
	%	:=
	%	\left\{
	%	\bar g\in\Cen_{R_F}(\F_{q^2})^\times:
	%	\bar gX\bar g^{-1}\in \F_{q^2}^\times X^{-1}
	%	\right\},
	%	\]
	%	and
	%	\[
	%	\mathcal A_F^-
	%	:=
	%	\left\{
	%	(\bar g,\bar g^{-1}\lambda X):
	%	\bar g\in\mathcal G_F^-,\ \lambda\in\F_{q^2}^\times
	%	\right\}.
	%	\]

	We can now determine the linear automorphism group.
	\begin{theorem}\label{thm:Aut-linear-skew}
		Let \(1\le k\le n-1\). If
		$
		(n,s,F)\neq (2,2,Y^2-\nu)$ for $\nu\in\F_q^\times$,
		then
		\[
		\Aut_{\E_F}(\Sk(F))=\mathcal A_F.
		\]
		%		\[
		%		\Aut_{\E_F}(\Sk(F))=\mathcal A_F^+.
		%		\]
		%		In the exceptional case
		%		\[
		%		n=s=2,\qquad F(Y)=Y^2-\nu,\qquad \nu\in\F_q^\times,
		%		\]
		%		one has
		%		\[
		%		\Aut_{\E_F}(\mathcal S_1(F))
		%		=
		%		\mathcal A_F^+\cup \mathcal A_F^-.
		%		\]
		%		Moreover, in this exceptional case the union is disjoint.
	\end{theorem}
	
	\begin{proof}
		We first show that
		\[
		\mathcal A_F\subseteq \Aut_{\E_F}(\Sk(F)).
		\]
		Let
		\[
		\bar g=\bar c\,\gamma X^t\in\mathcal G_F,
		\]
		with
		\(\bar c\in\E_F^\times\), \(\gamma\in\F_{q^n}^\times\), and
		\(t\in\{0,\dots,n-1\}\). For
		\[
		\overline a
		=
		\sum_{i=0}^{sk-1}\alpha_iX^i+RF(X^n)\in\Sk(F),
		\]
		set
		\[
		\overline a'
		:=
		\sum_{i=0}^{sk-1}
		\gamma\,\sigma^t(\alpha_i)\,\sigma^i(\gamma)^{-1}X^i
		+
		RF(X^n).
		\]
		Then \(\overline a'\in\Sk(F)\), and a direct computation gives
		\[
		\bar g\,\overline a=\overline a'\,\bar g.
		\]
		Hence
		\begin{equation}\label{eq:ga=a'g}
			\bar g\,\Sk(F)=\Sk(F)\,\bar g.	
		\end{equation}
		Therefore, for every \(\lambda\in\F_{q^n}^\times\),
		\[
		\bar g\,\Sk(F)\,\bar g^{-1}\lambda
		=
		\Sk(F)\lambda
		=
		\Sk(F).
		\]
		Thus
		\[
		(\bar g,\bar g^{-1}\lambda)\in\Aut_{\E_F}(\Sk(F)),
		\]
		and hence
		\[
		\mathcal A_F\subseteq \Aut_{\E_F}(\Sk(F)).
		\]
		
		Assume now that \(1\le k\le n/2\), and let
		\[
		(\bar g,\bar h)\in\Aut_{\E_F}(\Sk(F)).
		\]
		By Lemma~\ref{lem:mapsincentraliser}, both \(\bar g\) and \(\bar h\) lie in
		\(\Cen_{R_F}(\F_{q^n})^\times\). By Proposition~\ref{prop:structuremaps}, under the assumption that $(n,s,F)\neq (2,2,Y^2-\nu)$, we have
		\[
		\bar gX\bar g^{-1}=\beta X+RF(X^n)
		\]
		for some \(\beta\in\F_{q^n}^\times\).
		Equivalently,
		\[
		\bar gX=\beta X\bar g.
		\]
		By Lemma~\ref{lem:G-description}, we get
		\[
		\bar g\in\mathcal G_F.
		\]
		Hence, as in the proof of \eqref{eq:ga=a'g},
		$
		\bar g\,\Sk(F)=\Sk(F)\,\bar g.$
		Since
		$
		\Sk(F)
		=
		\bar g\,\Sk(F)\,\bar h
		=
		\Sk(F)\,\bar g\bar h,$
		we obtain
		\[
		\bar g\bar h\in \rid(\Sk(F))^\times.
		\]
		By Proposition~\ref{prop:invariantsSNsk},
		\[
		\rid(\Sk(F))=\F_{q^n}.
		\]
		Therefore there exists \(\lambda\in\F_{q^n}^\times\) such that
		\[
		\bar g\bar h=\lambda+RF(X^n),
		\]
		and hence
		$
		\bar h=\bar g^{-1}\lambda$.
		Thus
		\[
		(\bar g,\bar h)\in\mathcal A_F.
		\]
		We have proved the theorem for \(1\le k\le n/2\). 
		
		If \(k>n/2\), then necessarily
		we are not in the exceptional case \(n=2\), since \(1\le k\le n-1\). So, assume finally that \(k>n/2\). By \cite[Proposition~5.9]{gomez2025adjoint}, the dual code 
		\(\mathcal C^\perp\) of $\C=\cS_k(F)$ is equivalent to a code of the form
		\[
		\widehat{\mathcal C}:=\mathcal S_{n,s,n-k}(\widehat F),
		\]
		where \(\widehat F\) denotes the monic reciprocal polynomial associated with \(F\). Since
		\(n-k\le n/2\), the previous part applies to \(\widehat{\mathcal C}\). In particular,
		\[
		\bigl|\Aut_{\E_{\widehat F}}(\widehat{\mathcal C})\bigr|
		=
		\bigl|\mathcal A_{\widehat{F}}\bigr|,
		\]
		where
		\[
		\mathcal A_{\widehat{F}}
		:=
		\left\{
		(\bar g,\bar g^{-1}\lambda)
		:
		\bar g\in \mathcal G_{\widehat{F}},\ \lambda\in \F_{q^n}^\times
		\right\},
		\]
		and
		\[
		\mathcal G_{\widehat{F}}
		=
		\left\{
		\bar c\,\gamma X^t
		:
		\bar c\in \E_{\widehat F}^\times,\ \gamma\in \F_{q^n}^\times,\ t\in\{0,\dots,n-1\}
		\right\}.
		\]
		Since equivalent codes have isomorphic linear automorphism groups, we obtain
		\[
		\bigl|\Aut_{\E_F}(\mathcal C)\bigr|
		=
		\bigl|\Aut_{\E_{\widehat F}}(\widehat{\mathcal C})\bigr|.
		\]
		Moreover, \(|\E_F^\times|=|\E_{\widehat F}^\times|=q^s-1\), so
		$
		|\mathcal G_F|=|\mathcal G_{\widehat{F}}|$ and hence
		$|\mathcal A_F|=|\mathcal A_{\widehat{F}}|$.
		Therefore
		\[
		\bigl|\Aut_{\E_F}(\mathcal C)\bigr|
		=
		\bigl|\Aut_{\E_{\widehat F}}(\widehat{\mathcal C})\bigr|
		=
		|\mathcal A_F|=|\mathcal A_{\widehat{F}}|.
		\]
		Since we already proved that \(\mathcal A_F\subseteq \Aut_{\E_F}(\mathcal C)\), it follows that
		\[
		\Aut_{\E_F}(\mathcal C)=\mathcal A_F.
		\]
		This completes the proof.
	\end{proof}

	\begin{remark}\label{rem:n=s=2.Y^2-nu}
		For $(n,s,F)=(2,2,Y^2-\nu)$ with irreducible $Y^2-\nu\in \F_q[Y]$, we have to consider one extra case in the proof of Theorem~\ref{thm:Aut-linear-skew}. According to Proposition~\ref{prop:structuremaps}, it is also possible that
		$\bar gX\bar g^{-1}=\beta X^{-1}+RF(X^2)$. One can follow the proof of Theorem~\ref{thm:Aut-linear-skew} to cover this case. However, the cyclic semifield $\cS_1(F)$ with $(n,s,F)=(2,2,Y^2-\nu)$ is isotopic to a special case of Hughes-Kleinfeld semifields; see \cite[Section 4]{brownpumplunsteele2018automorphisms} and \cite[Section 4.1]{sheekey2020new}. In particular, its autotopism group has already been determined by Sandler in \cite[Theorem 3]{sandler1962autotopism}. Therefore, in this paper, we do not have to consider this special case with $(n,s,F)=(2,2,Y^2-\nu)$.
	\end{remark}
	
	As a consequence, one obtains a closed formula for the order of the linear automorphism group.
	\begin{corollary}\label{cor:Aut-linear-skew-order}
		Let \(1\le k\le n-1\). If
		$
		(n,s,F)\neq (2,2,Y^2-\nu)$ for $\nu\in\F_q^\times$,
		then
		\[
		|\Aut_{\E_F}(\Sk(F))|
		=
		n\,\frac{(q^s-1)(q^n-1)^2}{q-1}.
		\]
	\end{corollary}
	
	\begin{proof}
		By
		Theorem~\ref{thm:Aut-linear-skew},
		\[
		\Aut_{\E_F}(\Sk(F))
		=
		\left\{
		(\bar g,\bar g^{-1}\lambda):
		\bar g\in\mathcal G_F,\ \lambda\in\F_{q^n}^\times
		\right\}.
		\]
		Hence
		\[
		|\Aut_{\E_F}(\Sk(F))|
		=
		|\mathcal G_F|\,(q^n-1).
		\]
		For each \(t\in\{0,\dots,n-1\}\), set
		\[
		\mathcal G_{F,t}
		:=
		\{\bar c\,\gamma X^t:
		\bar c\in\E_F^\times,\ \gamma\in\F_{q^n}^\times\}.
		\]
		The decomposition
		\[
		\mathcal G_F
		=
		\bigsqcup_{t=0}^{n-1}\mathcal G_{F,t}
		\]
		is disjoint. Moreover,
		\[
		\E_F^\times\cap \F_{q^n}^\times=\F_q^\times.
		\]
		Therefore, for each fixed \(t\), the map
		\[
		\E_F^\times\times \F_{q^n}^\times
		\longrightarrow
		\mathcal G_{F,t},
		\qquad
		(\bar c,\gamma)\longmapsto \bar c\,\gamma X^t,
		\]
		has kernel
		\[
		\{(\mu,\mu^{-1}):\mu\in\F_q^\times\}.
		\]
		Thus
		\[
		|\mathcal G_{F,t}|
		=
		\frac{|\E_F^\times|\,|\F_{q^n}^\times|}{|\F_q^\times|}
		=
		\frac{(q^s-1)(q^n-1)}{q-1}.
		\]
		Summing over \(t\), we obtain
		\[
		|\mathcal G_F|
		=
		n\,\frac{(q^s-1)(q^n-1)}{q-1}.
		\]
		Consequently,
		\[
		|\Aut_{\E_F}(\Sk(F))|
		=
		|\mathcal G_F|\,(q^n-1)
		=
		n\,\frac{(q^s-1)(q^n-1)^2}{q-1}. \qedhere
		\]
	\end{proof}
	
	Our goal is now to determine the full automorphism group of \(\Sk(F)\). To this end, we need to relate the skew-polynomial model to the semilinear one. Recall that, by
	Theorem~\ref{th:polydescription}, if \(w_F\in \LL^\times\) is such that \(\N_{\LL/\F}(w_F)\) is a root of \(F\), then the code \(\mathcal C_k(F)\subseteq \End_{\E}(V)\) is obtained as the image of \(\Sk(F)\subseteq R_F\) under the algebra isomorphism
	$
	\Psi_{w_F}:R_F\longrightarrow \End_{\E}(V)$. As an immediate consequence of Theorem~\ref{thm:Aut-linear-skew}, we obtain the corresponding description of the linear automorphism group in the semilinear model. We are therefore in a position to determine the full automorphism group of \(\mathcal C_k(F)\), namely
	\[
	\Aut(\Sk(F)) \cong \Aut(\mathcal C_k(F))
	:=
	\left\{
	(M,N,\rho)\in \GL_{\E}(V)\times \GL_{\E}(V)\times \Aut_{\F_p}(\E):
	M\,\mathcal C_k(F)^\rho\,N=\mathcal C_k(F)
	\right\}.
	\]
	The additional semilinear symmetries are measured by the following subgroup of
	\(\Aut_{\F_p}(\F)\).
	
	\begin{definition}\label{def:HF}
		Let \(u_F\in\F\) be a root of \(F\). Set
		\[
		\mathcal H_F^+
		:=
		\left\{
		\rho\in \Aut_{\F_p}(\F):
		\rho(u_F)=\lambda u_F
		\text{ for some }\lambda\in \F_q^\times
		\right\}.
		\]
		If \(n=2\), set also
		\[
		\mathcal H_F^-
		:=
		\left\{
		\rho\in \Aut_{\F_p}(\F):
		\rho(u_F)=\lambda u_F^{-1}
		\text{ for some }\lambda\in \F_q^\times
		\right\}.
		\]
		Finally, define
		\[
		\mathcal H_F
		:=
		\begin{cases}
			\mathcal H_F^+, & \text{if } n>2,\\[1mm]
			\mathcal H_F^+\cup \mathcal H_F^-, & \text{if } n=2.
		\end{cases}
		\]
	\end{definition}
	\begin{proposition}\label{prop:field-part-aut}
		Let \(\rho\in \Aut_{\F_p}(\F)\). Then the following are equivalent:
		\begin{enumerate}
			\item[\rm (i)] There exist \(M,N\in \GL_{\E}(V)\) such that
			\[
			M\mathcal C_k(F)^\rho N=\mathcal C_k(F);
			\]
			
			\item[\rm (ii)] \(\rho\in \mathcal H_F\).
		\end{enumerate}
	\end{proposition}
	
	\begin{proof}
		This is exactly the special case \(G=F\) of
		Theorem~\ref{thm:comparison-F-G-semilinear} when \(1 \leq k \leq n/2\), with the fixed root \(u_F\) of \(F\). The remaining range \(n/2 < k \leq n-1\) follows by arguing exactly as in
		\Cref{thm:comparison-skew-model-semilinear}.
	\end{proof}
	
	\begin{proposition}\label{prop:size-HF}
		The set \(\mathcal H_F\) is a cyclic subgroup of \(\Aut_{\F_p}(\F)\). Moreover,
		\[
		|\mathcal H_F^+|\mid r\,\gcd(s,q-1),
		\]
		and
		\[
		|\mathcal H_F|
		\mid
		\varepsilon_n\, r\,\gcd(s,q-1),
		\qquad
		\varepsilon_n=
		\begin{cases}
			1, & n>2,\\
			2, & n=2.
		\end{cases}
		\]
		More precisely, if \(n=2\), then either \(\mathcal H_F^-=\varnothing\), in which
		case
		\[
		\mathcal H_F=\mathcal H_F^+,
		\]
		or \(\mathcal H_F^-\neq\varnothing\), in which case, for every
		\(\rho_0\in\mathcal H_F^-\),
		\[
		\mathcal H_F^-=\rho_0\mathcal H_F^+.
		\]
		Consequently,
		\[
		|\mathcal H_F|\in
		\left\{
		|\mathcal H_F^+|,\,
		2|\mathcal H_F^+|
		\right\}.
		\]
		If, in addition, \(n=s=2\), then \(\mathcal H_F^-\neq\varnothing\).
	\end{proposition}
	
	\begin{proof}
		It is easy to check that \(\mathcal H_F^+\) is a subgroup of
		\(\Aut_{\F_p}(\F)\). Consider the natural action of
		\(\F_q^\times\rtimes\Gal(\F_q/\F_p)\) on monic polynomials of degree \(s\)
		given by
		\[
		(\lambda,\tau)\cdot P(Y)
		=
		\lambda^s P^\tau(\lambda^{-1}Y).
		\]
		Then the stabilizer of \(F\) is
		\[
		\Stab^+(F)
		:=
		\left\{
		(\lambda,\tau)\in \F_q^\times\rtimes\Gal(\F_q/\F_p):
		F^\tau(Y)=\lambda^sF(\lambda^{-1}Y)
		\right\}.
		\]
		The map
		\[
		\mathcal H_F^+\longrightarrow \Stab^+(F),
		\qquad
		\rho\longmapsto (\lambda,\rho|_{\F_q}),
		\]
		where \(\lambda\) is determined by \(\rho(u_F)=\lambda u_F\), is a bijection.
		Indeed, applying \(\rho\) to the relation \(F(u_F)=0\) gives
		\[
		F^{\rho|_{\F_q}}(\lambda u_F)=0,
		\]
		and hence
		\[
		F^{\rho|_{\F_q}}(Y)=\lambda^sF(\lambda^{-1}Y).
		\]
		Conversely, if \((\lambda,\tau)\in\Stab^+(F)\), then \(\lambda u_F\) is a root
		of \(F^\tau\), and therefore there is a unique automorphism
		\(\rho\in\Aut_{\F_p}(\F)\) extending \(\tau\) and satisfying
		\[
		\rho(u_F)=\lambda u_F.
		\]
		Thus
		\[
		|\mathcal H_F^+|=|\Stab^+(F)|.
		\]
		Since \(\Stab^+(F)\) is a subgroup of
		\(\F_q^\times\rtimes\Gal(\F_q/\F_p)\), its order divides \(r(q-1)\). On the
		other hand, \(\mathcal H_F^+\) is a subgroup of the cyclic group
		\(\Aut_{\F_p}(\F)\), whose order is \(rs\). Hence
		\[
		|\mathcal H_F^+|\mid \gcd(r(q-1),rs)
		=
		r\,\gcd(s,q-1).
		\]
		
		If \(n>2\), then \(\mathcal H_F=\mathcal H_F^+\), and the assertion follows.
		Assume from now on that \(n=2\). If \(\mathcal H_F^-=\varnothing\), then again
		\(\mathcal H_F=\mathcal H_F^+\).
		
		Suppose therefore that \(\mathcal H_F^-\neq\varnothing\), and fix
		\(\rho_0\in\mathcal H_F^-\). Write
		\[
		\rho_0(u_F)=\lambda_0u_F^{-1},
		\qquad
		\lambda_0\in\F_q^\times.
		\]
		We claim that
		\[
		\mathcal H_F^-=\rho_0\mathcal H_F^+.
		\]
		Let \(\eta\in\mathcal H_F^+\), say
		\[
		\eta(u_F)=\mu u_F,
		\qquad
		\mu\in\F_q^\times.
		\]
		Then
		\[
		(\rho_0\eta)(u_F)
		=
		\rho_0(\mu)\lambda_0u_F^{-1},
		\]
		with \(\rho_0(\mu)\lambda_0\in\F_q^\times\). Hence
		\(\rho_0\eta\in\mathcal H_F^-\), and so
		\[
		\rho_0\mathcal H_F^+\subseteq\mathcal H_F^-.
		\]
		Conversely, let \(\rho\in\mathcal H_F^-\), say
		\[
		\rho(u_F)=\lambda u_F^{-1},
		\qquad
		\lambda\in\F_q^\times.
		\]
		Set \(\eta:=\rho_0^{-1}\rho\). Since
		\[
		\rho_0^{-1}(u_F)=\rho_0^{-1}(\lambda_0)u_F^{-1},
		\]
		we obtain
		\[
		\eta(u_F)
		=
		\rho_0^{-1}(\lambda)\,
		\rho_0^{-1}(\lambda_0)^{-1}u_F.
		\]
		Thus \(\eta\in\mathcal H_F^+\), and consequently
		\[
		\rho=\rho_0\eta\in\rho_0\mathcal H_F^+.
		\]
		This proves the claim.
		
		Furthermore,
		\[
		\rho_0^2(u_F)
		=
		\rho_0(\lambda_0)\lambda_0^{-1}u_F,
		\]
		and therefore \(\rho_0^2\in\mathcal H_F^+\). Since
		\(\Aut_{\F_p}(\F)\) is abelian, it follows that
		\[
		\mathcal H_F
		=
		\mathcal H_F^+\cup\rho_0\mathcal H_F^+
		=
		\langle \mathcal H_F^+,\rho_0\rangle
		\]
		is a cyclic subgroup of \(\Aut_{\F_p}(\F)\). 
		%		Hence \(\mathcal H_F\) is cyclic, being	a subgroup of the cyclic group \(\Aut_{\F_p}(\E)\).
		
		The index of \(\mathcal H_F^+\) in \(\mathcal H_F\) is at most two. Therefore
		\[
		|\mathcal H_F|\in
		\left\{
		|\mathcal H_F^+|,\,
		2|\mathcal H_F^+|
		\right\}.
		\]
		In particular,
		\[
		|\mathcal H_F|\mid 2r\,\gcd(s,q-1).
		\]
		
		It remains to prove the final assertion. Assume \(n=s=2\). Then
		\(\F=\F_{q^2}\), and the \(q\)-Frobenius automorphism
		\[
		\varphi_q:x\longmapsto x^q
		\]
		is the nontrivial element of \(\Gal(\F/\F_q)\). Since \(u_F\) has degree two
		over \(\F_q\), its conjugates over \(\F_q\) are \(u_F\) and \(u_F^q\). Hence
		\[
		u_Fu_F^q=\N_{\F/\F_q}(u_F)\in\F_q^\times,
		\]
		and therefore
		\[
		\varphi_q(u_F)
		=
		u_F^q
		=
		\N_{\F/\F_q}(u_F)\,u_F^{-1}.
		\]
		Thus \(\varphi_q\in\mathcal H_F^-\), and so
		\[
		\mathcal H_F^-\neq\varnothing. \qedhere
		\]
	\end{proof}

	We now obtain the structure of the full automorphism group.
	
	\begin{theorem}\label{thm:full-semilinear-aut}
		The projection
		\[
		\pi:\Aut(\mathcal C_k(F))\longrightarrow \Aut_{\F_p}(\F),
		\qquad
		(M,N,\rho)\longmapsto \rho
		\]
		has image \(\mathcal H_F\), and its kernel is canonically isomorphic to
		\(\Aut_{\E}(\mathcal C_k(F))\). Hence there is a short exact sequence
		\[
		1\longrightarrow \Aut_{\E}(\mathcal C_k(F))
		\longrightarrow \Aut(\mathcal C_k(F))
		\overset{\pi}{\longrightarrow}
		\mathcal H_F
		\longrightarrow 1.
		\]
	\end{theorem}
	
	\begin{proof}
		By definition, a triple \((M,N,\rho)\) belongs to
		\(\Aut(\mathcal C_k(F))\) if and only if
		\[
		M\mathcal C_k(F)^\rho N=\mathcal C_k(F).
		\]
		Hence, for \(\rho\in\Aut_{\F_p}(\F)\), the condition
		\(\rho\in\im(\pi)\) is equivalent to the existence of
		\(M,N\in\GL_{\E}(V)\) such that
		$
		M\mathcal C_k(F)^\rho N=\mathcal C_k(F).
		$
		By Proposition~\ref{prop:field-part-aut}, this is equivalent to
		\(\rho\in\mathcal H_F\). Therefore
		\[
		\im(\pi)=\mathcal H_F.
		\]
		
		It remains to identify the kernel. We have
		\[
		(M,N,\rho)\in \ker(\pi)
		\quad\Longleftrightarrow\quad
		\rho=\id_{\E}.
		\]
		Since \((M,N,\rho)\in\Aut(\mathcal C_k(F))\), this is equivalent to
		\[
		M\mathcal C_k(F)N=\mathcal C_k(F).
		\]
		Thus
		\[
		\ker(\pi)
		=
		\{(M,N,\id_{\E}):(M,N)\in \Aut_{\E}(\mathcal C_k(F))\}.
		\]
		Consequently, the map
		\[
		\Aut_{\E}(\mathcal C_k(F))\longrightarrow \ker(\pi),
		\qquad
		(M,N)\longmapsto (M,N,\id_{\E})
		\]
		is a canonical group isomorphism. This proves the claim.
	\end{proof}
	
	\begin{corollary}\label{cor:full-semilinear-aut-order}
		When $(n,s,F)\neq (2,2,Y^2-\nu)$ for $\nu\in \F_q^\times$, 
		\[
		|\Aut(\mathcal C_k(F))|
		=
		|\mathcal H_F|\cdot |\Aut_{\E}(\mathcal C_k(F))|.
		\]
		In particular,
		\[
		|\Aut(\mathcal C_k(F))|
		=
		|\mathcal H_F|\,
		n\,\frac{(q^s-1)(q^n-1)^2}{q-1}.
		\]
	\end{corollary}
	
	\begin{proof}
		The first identity follows from the short exact sequence in
		Theorem~\ref{thm:full-semilinear-aut}. The second one follows from
		Corollary~\ref{cor:Aut-linear-skew-order}.
	\end{proof}
	
	\begin{remark}
		The subgroup \(\mathcal H_F\) measures precisely the additional symmetries induced by
		field automorphisms of \(\F\). In particular, the full semilinear automorphism group
		coincides with its \(\E\)-linear subgroup, under the canonical embedding, if and only if
		\(\mathcal H_F\) is trivial; equivalently,
		\[
		\Aut(\mathcal C_k(F))=\Aut_{\E}(\mathcal C_k(F))
		\qquad\Longleftrightarrow\qquad
		\mathcal H_F=1.
		\]
	\end{remark}
	
	We now turn to the special case \(k=1\), where \(\mathcal C_k(F)\) is the spread set of the cyclic
	semifield \(\mathbb S(F)\). In this situation, semilinear automorphisms of the code correspond
	precisely to autotopisms of the associated semifield. Consequently, the preceding results translate
	directly into a description of the full autotopism group of cyclic semifields.
	
	\begin{corollary}\label{cor:full-autotopism-semifield}
		Assume that \(k=1\) and $(n,s,F)\neq (2,2,Y^2-\nu)$ for $\nu\in \F_q^\times$. Let \(\mathbb S(F)\) be the cyclic semifield associated with \(F\). Then
		\[
		|\Aut(\mathbb S(F))|
		=
		|\mathcal H_F|\,
		n\,\frac{(q^s-1)(q^n-1)^2}{q-1}.
		\]
	\end{corollary}
	
	\begin{proof}
		Since \(k=1\), the code \(\mathcal C_k(F)\) is the spread set of
		\(\mathbb S(F)\). Hence the full semilinear automorphism group of
		\(\mathcal C_k(F)\) identifies with the autotopism group of \(\mathbb S(F)\).
		The assertion follows from Corollary~\ref{cor:full-semilinear-aut-order}.
	\end{proof}
	
	\section*{Acknowledgments}
	Paolo Santonastaso is partially supported by the Italian National Group for Algebraic and Geometric Structures and their Applications (GNSAGA- INdAM). Yue Zhou is partially supported by the National Natural Science Foundation of China (No.\ 12371337) and the Natural Science Foundation of Hunan Province (No.\ 2023RC1003).
	
	Paolo Santonastaso and Yue Zhou are deeply grateful to the Kaiyuan International Mathematical Sciences Institute (KIMSI) in Changsha, China, for hosting their visit and providing office space and facilities, where the major part of this work was completed.

	\bibliographystyle{abbrv}
	\bibliography{biblio}
	
\end{document}